\documentclass[12pt]{article}

\usepackage{amsmath,amssymb,amsfonts,amsthm}
\usepackage{enumerate,enumitem}
\usepackage{cite}
\usepackage{mathrsfs}
\usepackage{showlabels}
\usepackage[english]{babel}
\usepackage[autostyle, english = american]{csquotes}
\MakeOuterQuote{"}

\setlist[enumerate,1]{label=(\arabic*)}
\setlist[enumerate,3]{label=(\roman*)}

\newtheorem{thm}{Theorem}[section]
\newtheorem{lem}[thm]{Lemma}

\newtheorem{prop}[thm]{Proposition}

\theoremstyle{definition}
\newtheorem{defn}[thm]{Definition}
\newtheorem*{question}{Question}

\newtheorem*{nota}{Notation}

\newcommand{\dom}[1]{\text{dom}(#1)}

\newcommand{\rest}{\upharpoonright}
\newcommand{\ang}[1]{\langle#1\rangle}
\renewcommand{\phi}{\varphi}

\newcommand{\forces}{\Vdash}

\renewcommand{\O}{\mathcal{O}}
\newcommand{\Deg}{\mathscr{D}}
\newcommand{\Degh}{\Deg_h}
\newcommand{\DegO}{\Deg_h(\leq_h\! \O)}
\newcommand{\omegaoneck}{\omega_1^{\text{CK}}}
\renewcommand{\a}{\textbf{a}}
\renewcommand{\b}{\textbf{b}}
\renewcommand{\c}{\textbf{c}}
\renewcommand{\d}{\textbf{d}}
\newcommand{\e}{\textbf{e}}

\newcommand{\Lat}{\mathscr{L}}
\newcommand{\hLat}{\hat{\Lat}}

\newcommand{\sqleq}{\sqsubseteq}

\newcommand{\join}{\sqcup}
\newcommand{\meet}{\sqcap}
\renewcommand{\Join}{\bigsqcup}

\newcommand{\M}{\mathscr{M}}
\newcommand{\U}{\mathscr{U}}
\newcommand{\V}{\mathscr{V}}

\newcommand{\Latrep}{\{\Theta_i:i\in\omega\}}

\newcommand{\Lang}[1][{}]{\mathcal{L}(\omegaoneck,\texttt{G}^{\texttt{#1}})}
\newcommand{\Model}[1][{}]{\mathcal{M}(\omegaoneck,\mathcal{G}^{#1})}
\newcommand{\concat}{{}^\smallfrown}
\renewcommand{\P}{\mathcal{P}}

\newcommand{\G}{\mathcal{G}}

\newcommand{\X}{\textbf{X}}
\newcommand{\condition}[1][{}]{(\Phi_{#1},\X_{#1})}

\title{The $\Sigma_2$ theory  of $\DegO$ as an uppersemilattice with least and greatest element is decidable}
\author{James Barnes}
\date{}

\begin{document}

\maketitle

\begin{abstract}
	We establish the decidability of the $\Sigma_2$ theory of $\Degh(\leq_h \O)$, the hyperarithmetic degrees below Kleene's $\O$, in the language of uppersemilattices with least and greatest element. This requires a new kind of initial-segment result and a new extension of embeddings result both in the hyperarithmetic setting.
\end{abstract}

\section*{Introduction}
Hyperarithmetic reducibility is as a notion of reduction with connections to Turing reducibility, recursive well-orderings, and definability in second-order arithmetic. A subset $X$ of $\omega$ is \textbf{hyperarithmetic} in a subset $Y$ of $\omega$, $X\leq_h Y$, if there is an ordinal $\delta$ with a $Y$-recursive representation such that $X$ is Turing reducible to the $\delta$th jump of $Y$. (Some care must be taken in defining $Y^{(\delta)}$: $Y^{(0)} = Y, Y^{(\delta+1)} = (Y^{\delta})'$, and if $\{\delta_{n}\}_{n=0}^\infty$ is a $Y$-recursive sequence of $Y$-recursive ordinals, then $Y^{(\lim \delta_n)} = \bigoplus_{n} Y^{(\delta_n)}$. One must show that, up to Turing degree, the definition of $Y^{(\delta)}$ for limit $\delta$ does not depend on the choice of $Y$-recursive cofinal sequence.) Kleene showed that $X$ is hyperarithmetic in $Y$ iff $X$ is $\Delta_1^1$ definable in second-order arithmetic with a predicate for membership in $Y$.

The \textbf{hyperarithmetic degrees} (or hyperdegrees), $\Degh$, is the quotient of $2^\omega$ under hyperarithmetical equivalence: $X\equiv_h Y$ iff $X\leq_h Y$ and $Y\leq_h X$. This degree structure shares many similarities with the Turing degrees, for instance, it is an uppersemilattice with least element: The join operator can be defined (on representatives for degrees) as
	\[
		X \oplus Y = \{2n : n\in X\} \cup \{2n + 1 : n\in Y\},
	\]
and the least element is the degree of the empty set. There is also a notion of jump: the \textbf{hyperjump} of $X$ is $\O^X$, the $\Pi_1^1(X)$ complete set of notations for $X$-recursive ordinals. This operator bears some similarity to the Turing jump operator, which takes a set $X$ and returns a complete $\Sigma^0_1(X)$ set. The reader may be tempted to an analogy between being recursively enumerable in $X$ (which is equivalent to being $\Sigma^0_1(X)$) and being $\Pi_1^1$ in $X$; however, this temptation will lead one astray: They only hyperdegrees with a $\Pi_1^1$ member are the trivial hyperdegree and the hyperdegree of $\O$, the hyperjump of the empty set.

Despite this, $\DegO$, the hyperdegrees less than $\O$, is still an interesting substructure of $\Degh$ in its own right and may be considered analogous to $\Deg_T(\leq_T 0')$, the Turing degrees below $0'$. It is natural to ask questions about what kind of structures embed into $\DegO$, what its initial segments look like, the complexity of its theory, and so on. This paper concerns lattice embeddings of finite lattices into $\DegO$ that takes the top element to $\O$ and everything else to an initial segment, extensions of embeddings of finite uppersemilattices, and an application of these facts to the complexity of the theory of $\DegO$.

Note that, although our discussion so far has concerned subsets of $\omega$, very little is changed by instead considering functions from $\omega$ to itself: One can identify a function from $\omega$ to $\omega$ with its graph, and so with a subset of $\omega$ by some fixed recursive pairing function, and one can identify a subset of $\omega$ with its characteristic function. In the following, we will not observe a distinction between these two perspectives.
	
\section{Decidability}

	\begin{defn}[Lattice and uppersemilattice]
		A \textbf{lattice} 
			\[
				\Lat=(L,\sqleq_\Lat, \meet_\Lat, \join_\Lat, \top_\Lat, \bot_\Lat)
			\]	
		is a structure such that $\sqleq_\Lat$ is a partial order on the set $L$ satisfying the following:
			\begin{itemize}
				\item $\top_\Lat$ is the greatest element of $\Lat$.
				\item $\bot_\Lat$ is the least element of $\Lat$.
				\item Each pair $x,y\in \Lat$ have a greatest lower bound $x\meet_\Lat y$.
				\item Each pair $x,y\in \Lat$ have a least upper bound $x\join_\Lat y$ in $\Lat$.
			\end{itemize}
	
		An \textbf{uppersemilattice} $\U=(U,\sqleq_\U, \join_\U, \bot_\U)$ is like a lattice, but there need not be a greatest element, nor need greatest lower bounds exist.
	
		An \textbf{uppersemilattice with $\top$} is an uppersemilattice with a greatest element $\top_\U$.
		\end{defn}
	
	\begin{nota}
		We will denote lattices and uppersemilattices with calligraphic roman letters $\Lat, \U,\V$ and the like. Elements of lattices (and uppersemilattices) will be denoted with lowercase roman letters from the end of the alphabet: $x,y,z,w$.
		
		Whenever confusion will not arise, we will drop the subscripts on the various parts of the structure, e.g., we will write $\sqleq$ instead of $\sqleq_\Lat$ if it is understood to which structure we are referring. Additionally, we will abbreviate uppersemilattice as USL and uppersemilattice with $\top$ as USL$^\top$.
	\end{nota}

	Note that a finite USL $\U$ is a lattice: its greatest element and meet are given by 
		\[
			\top := \Join_{x\in\U} x \quad \text{ and } \quad x\meet y := \Join_{z\sqleq x,y} z, \quad \text{ respectively}.
		\]
	As $\U$ is finite and has least element $\bot$, each of these joins is nonempty and finite; therefore, they are well defined.

	The satisfiability of a $\Pi_2$ sentence in the language of USL$^\top$s over $\DegO$ can be effectively reduced (in a manner entirely analogous to that in Lerman Theorem VII.4.4 \cite{Lerman1983}) to deciding questions of the following form:
		\begin{question}
			Let $\U$ be a finite USL$^\top$ and let $\V_1,\ldots,\V_n$ be finite USL$^\top$s each a superstructure of $\U$. Then is it the case that every embedding of $\U$ into $\DegO$ extends to an embedding of one of the $\V_i$s?
		\label{keyquestion}
		\end{question}
	Note that the embedding must preserve $\join$, map $\bot$ to the degree of the empty set, and $\top$ to the degree of $\O$.

	To answer the question, we introduce some terminology and state our results.
		\begin{defn}[Almost-initial-segment]
			If $\U$ and $\V$ are USL$^{\top}$s and $\U$ is a substructure of $\V$, then $\U$ is an \textbf{almost-initial-segment} of $\V$ (we also say $\V$ is an \textbf{almost-end-extension} of $\U$) if whenever $u\in \U$ and $v\in \V\setminus\U$ satisfy $v\sqleq u$, then $u = \top$, i.e., the only way an element of $\U$ is above something strictly in $\V$ is if it is the greatest element (in both $\U$ and $\V$).
		\end{defn}

		\begin{thm}
			Let $\Lat$ be a finite lattice. Then there is a lattice embedding $f:\Lat\rightarrow\DegO$ such that the image of $\Lat$ under $f$ is an almost-initial-segment of $\DegO$.
		\label{maintheorem1}
		\end{thm}

		\begin{thm}
			Let $\U$ and $\V$ be finite USL$^\top$s such that $\U$ is an almost initial segment of $\V$. Then every embedding of $\U$ into $\DegO$ extends to one of $\V$.
		\label{maintheorem2}
		\end{thm}
	Given these results our aforementioned question can be answered by determining whether any of the $\V_i$s are almost-end-extensions of $\U$, which we can answer in a uniform and recursive manner: If $\V_i$ is an almost-end-extension of $\U$, then Theorem \ref{maintheorem2} provides an extension of any embedding of $\U$ into $\DegO$ to one of $\V_i$. On the other hand, if no $\V_j$ is an almost-end-extension of $\U$, then Theorem \ref{maintheorem1} provides an embedding of $\U$ into $\DegO$ as an almost-initial-segment, which can not extend to any of the $\V_j$. 
	
	The major project for the rest of this paper is to establish Theorems \ref{maintheorem1} and \ref{maintheorem2}. We proceed with Theorem \ref{maintheorem1}.
	
\section{Almost-initial-segments of $\DegO$}

Recall that a finite USL$^\top$ is a lattice. An arbitrary embedding of a USL$^\top$ into $\DegO$ need not preserve the meet structure; however, if the image of our embedding is an almost-initial-segment, then the meet structure will be preserved. Consequently, while our discussion concerns almost-initial-segments of $\DegO$, there is no loss of generality in considering lattices.

Kjos-Hanssen and Shore\cite{KjoshanssenShore2010} have produced initial segment embeddings of every sublattice of every hyperarithmetic lattice, a class which contains all finite lattices. We combine their forcing construction with ideas from Lerman and Shore\cite{LermanShore1988} to code $\O$ into the top real we construct, while preserving the lattice structure in such a way that everything except $\top$ is mapped into some initial segment.

Let $\Lat$ be a finite lattice and let $A$ be the set of coatoms of $\Lat$, i.e.,
	\[
		A = \{x\in \Lat : x\sqsubset\top\text{ and there is no $y\in\Lat$ s.t. } x \sqsubset y \sqsubset 1\}.
	\]
Let $f$ be an embedding of $\Lat$ into $\Degh$ as an initial-segment, as provided by Kjos-Hanssen and Shore \cite{KjoshanssenShore2010}. As our lattice is finite, it is recursive; therefore, we can choose such an $f$ with takes $\top$ to a degree below that of $\O$. We define a map $\tilde{f}:\Lat\rightarrow\DegO$ by
	\[
		\tilde{f}(x) = \begin{cases}
			f(x)&\text{if $x\neq \top$},\\
			\text{degree}(\O)&\text{if $x = \top$.}
		\end{cases}
	\]
This map will not, in general, be a lattice embedding of $\Lat$, as we may no longer preserve the join structure. However, if $|A|<2$ then there are no $x,y\in\Lat$ such that $x,y<\top$ yet $x\join y = \top$. In this case, $\tilde{f}$ will be a lattice embedding, and, by choice of $f$, will be an almost-initial-segment embedding of $\Lat$ into $\DegO$. (Indeed, we could send $\top$ to any degree above $f(\top)$ and still have a lattice embedding where the image of $\Lat\setminus\{\top\}$ is an initial segment). Henceforth, we assume that $|A|\geq 2$.

The rest of this section approximately follows the structure of Kjos-Hanssen and Shore\cite{KjoshanssenShore2010} with the additional concern of coding $\O$ into the image of the top element of a lattice. Our notion of forcing is rather simpler than theirs (because we are only concerned with finite lattices), but whenever we wish to meet a dense set of a particular kind, we need to show we can do so without interfering with our coding procedure (which we are yet to define).

\subsection{Lattice representations}

To define our notion of forcing for almost-initial-segments, we require a strong kind of representation of $\Lat$.

\begin{defn}[USL table]
	A set $\Theta$ of maps from $\Lat$ to $\omega$ is an \textbf{USL table} for $\Lat$ if it has the following properties:
	\begin{enumerate}
		\item(nontriviality of $\Theta$) The zero map $x \mapsto 0$ is in $\Theta$ (we denote this map by $0$ as well).
		\item($\bot$ is trivial) For every $\alpha\in \Theta$, $\alpha(\bot) = 0$.
	\end{enumerate}
	And for every choice $x,y,z\in \Lat$:
	\begin{enumerate}[resume]
		\item(Order) If $x\sqleq y$, and $\alpha,\beta\in \Theta$ satisfy $\alpha(y) = \beta(y)$, then $\alpha(x) =\beta(x)$.
		\item(Differentiation) If $x\not\sqleq y$, then there are $\alpha,\beta\in\Theta$ such that $\alpha(y)=\beta(y)$ yet $\alpha(x)\neq \beta(x)$.
		\item(Join) If $x\join y = z$ and $\alpha,\beta\in \Theta$ satisfy $\alpha(x) = \beta(x)$ and $\alpha(y)= \beta(y)$, then $\alpha(z)=\beta(z)$.
	\end{enumerate}
\end{defn}

\begin{nota}
	We will denote lattice tables by $\Theta,\Theta_1,\Theta_2$ and so on, and their elements will be denoted by lowercase Greek letters $\alpha,\beta$ and $\gamma$.
	
	For $x\in \Lat$ and $\alpha,\beta\in \Theta$ members of an USL table for $\Lat$, we write $\alpha\equiv_{x} \beta$ if $\alpha(x)=\beta(x)$, which is clearly an equivalence relation on $\Theta$. We write $\alpha\equiv_{x,y}\beta$ to mean that $\alpha\equiv_{x} \beta$ and $\alpha \equiv_{y} \beta$.
	
	We extend this notation to partial functions (and so, in particular, to strings) by declaring $f \equiv_{x} g$ if wherever $f,g$ are both defined their values agree modulo $x$.
\end{nota}

\begin{defn}[Sequential lattice representation]
	A nested sequence $\{\Theta_i : i\in \omega\}$ of finite USL tables for $\Lat$ is a \textbf{sequential (lattice) representation} for $\Lat$ if for every choice of $x,y,z\in\Lat$ and $i\in \omega$:
	\begin{enumerate}
		\item If $x\meet y = z$, then there are \textbf{meet interpolants} for $\Theta_i$ in $\Theta_{i+1}$, i.e., if $\alpha,\beta\in \Theta_i$ and $\alpha\equiv_{z} \beta$, then there are $\gamma_0,\gamma_1,\gamma_2\in \Theta_{i+1}$ such that 
			\[
				\alpha \equiv_{x} \gamma_0 \equiv_{y} \gamma_1 \equiv_{x} \gamma_2 \equiv_{y} \beta.
			\]
		\item There are \textbf{homogeneity interpolants} for $\Theta_i$ in $\Theta_{i+1}$, i.e., for all $\alpha_0,\alpha_1,\beta_0,\beta_1 \in \Theta_i$ such that
		\[
		\left(\forall x\in \Lat\right)[\alpha_0\equiv_x \alpha_1 \rightarrow \beta_0 \equiv_x \beta_1]
		\]
		there are $\gamma_0,\gamma_1\in \Theta_{i+1}$ and \textbf{$\Lat$-homomorphisms} $f,g,h:\Theta_i\rightarrow \Theta_{i+1}$ such that
		\[
		f:\alpha_0,\alpha_1\mapsto \beta_0, \gamma_1, \quad g:\alpha_0,\alpha_1\mapsto \gamma_0,\gamma_1, \quad h:\alpha_0,\alpha_1\mapsto \gamma_0,\beta_1.
		\]
		($f$ is an $\Lat$-homomorphism from $\Theta_i$ to $\Theta_{i+1}$ if for each $\alpha,\beta\in\Theta_i$ and each $x\in\Lat$ if $\alpha\equiv_x \beta$, then $f(\alpha)\equiv_x f(\beta)$.)
	\end{enumerate}
\end{defn}

The above definition is a simplification of the representation given in Theorem 5.1 of Kjos-Hanssen and Shore\cite{KjoshanssenShore2010}. Ours is simpler because our lattices are finite, and so we do not need to approximate our lattice with a growing sequence of finite lattices. Using such a representation, we could embed $\Lat$ as an initial segment of $\Degh$, or even $\DegO$; however, we would have insufficient control over the image of $\top$. We introduce apparatus that allows us to code $\O$ into the image of $\top$:

\begin{defn}[Coding-ready representation]
	A sequential representation $\{\Theta_i : i\in \omega\}$ has $C\subseteq \Theta_0$ as a \textbf{coding set} if there is a bijective map $g$ from 
		\[
			\left\{\ang{x,y,k}:  k\in\{0,1\},\, x\join y =\top,\, \text{and}\ x,y\neq \top\right\}
		\] 
	to $C$ such that:
	\begin{enumerate}\setcounter{enumi}{2}
		\item For every $x,y\in\Lat\setminus\{\top\}$ if $x \join y = \top$, then
			\[
				g(x,y,0) \equiv_x g(x,y,1) \text{ and } g(x,y,0)\not\equiv_y g(x,y,1).
			\]
		\item For all $\alpha\in C$ and all $x\in \Lat\setminus\{\top\}$, there is a $\beta\in \Theta_0\setminus C$ such that $\alpha\equiv_x\beta.$
	\end{enumerate}
	
	The table is \textbf{acceptable for $A$} (the set of coatoms of $\Lat$) if for each $i > 0$ there is a subset $\Theta^\ast_i$ of $\Theta_i$, containing $\Theta_{i-1}$ such that (1) holds for $\Theta^\ast_{i+1}$ in place of $\Theta_{i+1}$, (2) holds for $\Theta^\ast_{i+1}$ in place of $\Theta_i$, and further, (in the notation of (2)) if $f(\alpha)\in C$, then $\alpha=\alpha_0$ or $\alpha=\alpha_1$ (and the same for $g$ and $h$), and, finally: 
	\begin{enumerate}[resume]
		\item For all $x\in A$, all $i\in\omega$, and all $\alpha_0,\alpha_1\in\Theta_i$ there exists $\beta_0,\beta_1\in \Theta_{i+1}^\ast\setminus \Theta_i$ such that
		\[
		\alpha_0\equiv_x\beta_0,\, \alpha_1 \equiv_{x} \beta_1,\,\text{and}\ (\forall y\in\Lat)[\alpha_0\equiv_{y} \alpha_1\rightarrow \beta_0\equiv_{y} \beta_1].
		\]
	\end{enumerate}
	If $\{\Theta_i : i\in\omega\}$ has coding set $C$, the differentiation property of USL tables is satisfied outside of $C$ (i.e., for each $x\not\sqleq y$ there are $\alpha,\beta\in\Theta_0\setminus C$ such that $\alpha\equiv_{y} \beta$ yet $\alpha \not\equiv_{x} \beta$), and it is acceptable for $A$, we call it \textbf{coding-ready}.
\end{defn}

We will construct a recursive coding-ready sequential representation $\{\Theta_i : i\in\omega\}$ for $\Lat$ shortly. Firstly, we motivate and explain the definition: our representation will be nested as displayed
	\[
		C \subsetneq \Theta_0 \subsetneq \Theta_1^\ast \subsetneq \Theta_1 \subseteq \Theta_2^\ast \subsetneq \Theta_2 \subsetneq \cdots
	\]
Each $\Theta_i$ will be a USL table for $\Lat$; and we can find meet interpolants for elements of $\Theta_i$ inside $\Theta_{i+1}^\ast$, and we can find homogeneity interpolants for $\Theta_{i+1}^\ast$ in $\Theta_{i+1}$. The elements of $C$ are special, and they indicate we are coding. (3) tells us that for each $x,y\in\Lat$ joining up nontrivially to $\top$ that there is a pair of unique coding elements. (4) and the fact that we can satisfy the differentiation property outside of $C$ tells us that we can replace a coding element with an element that does not code and still preserve a congruence. (5) in the definition of acceptable for $A$ tells us that if we can find a split (this will be defined later), then we can find a split not using coding elements, and the requirement that $f,g,$ and $h$ only take on coding values when entirely necessary (i.e. when $\beta_0\in C$ or $\beta_1\in C$) allows us to find homogeneity interpolants which do no coding.

\begin{thm}
	Let $\Lat$ be a finite lattice with at least two coatoms and $A$ be the set of coatoms of $\Lat$. Then there is a recursive coding-ready sequential lattice representation for $\Lat$.\label{latticerepexistence}
\end{thm}

Lerman and Shore\cite{LermanShore1988} construct a sequential representation which is very similar to ours; the main difference is in the homogeneity interpolants. In our notation the $\Lat$-homomorphisms $f,g,h$ in Lerman and Shore act on $\alpha_0,\alpha_1$ as follows:
	\[
		f:\alpha_0,\alpha_1\mapsto \beta_0,\gamma_0,\quad g:\alpha_0,\alpha_1\mapsto \gamma_0,\gamma_1, \quad h:\alpha_0,\alpha_1 \mapsto \gamma_1,\beta_1.
	\]
In particular, $f(\alpha_1) = \gamma_0$ and $h(\alpha_0) = \gamma_1$ instead of $\gamma_1$ and $\gamma_0$, respectively, which is what we required in (2) of the definition of sequential representation. (There is also a difference in the coding set; Lerman and Shore have coding elements for each unordered pair $\{x,y\}$ and we have them for each ordered pair $\ang{x,y}$. The reason for this difference is purely notational, and does not present any mathematical difficulties.)

In their construction Lerman and Shore begin with a finite USL table $\Theta$ for $\Lat$ and then construct a finite USL table extension $\Theta_0$ of $\Theta$ and observe that $\Theta_0$ contains a coding set $C$ disjoint from $\Theta$ satisfying properties (3) and (4). Furthermore, as we started with a USL table $\Theta$, the differentiation property is satisfied outside $C$, as required.

Then they proceed inductively: Given $\Theta_i$, by Lerman Appendix B.2.6 \cite{Lerman1983}, there is a finite USL table $\Theta_i^1$ which contains meet interpolants for $\Theta_i$. They then argue that $\Theta_i^1$ can be extended to a finite USL table $\Theta_i^\ast$ satisfying (5), and as $\Theta_i\subseteq \Theta_i^1\subsetneq \Theta_{i+1}^\ast$ we can find meet interpolants for $\Theta_i$ in $\Theta_{i+1}^\ast$. All of this is uniformly recursive.

Consequently, all we need to show is that given $\Theta_{i+1}^\ast$ a finite USL table for $\Lat$ that we can find (uniformly and recursively) a finite USL table extension $\Theta_{i+1}$ of $\Theta_{i+1}^\ast$ such that we can find homogeneity interpolants for $\Theta_{i+1}^\ast$ in $\Theta_i$ and the $\Lat$-homomorphisms avoid the coding set (unless $\beta_0\in C$ or $\beta_1\in C$).

Kjos-Hanssen and Shore \cite{KjoshanssenShore2010} have already completed this work for us: Their Proposition 5.6 (taking $\hLat = \Lat$) says that we can find such a USL table extension which has homogeneity interpolants of the kind we need (again uniformly and recursively). An examination of their proof shows that if $\alpha\in \Theta_{i+1}^\ast$ then $f(\alpha)\notin \Theta_{i+1}^\ast$ unless $\alpha = \alpha_0$, and so $f(\alpha)=f(\alpha_0) = \beta_0$, and similarly for $g$ and $h$. Hence applying this Proposition allows us to continue our induction, and completes the construction.

\subsection{Perfect trees and forcing}

From here onward we fix a finite lattice $\Lat$ with a set of coatoms $A$ of cardinality at least two and fix a recursive coding-ready sequential representation $\{\Theta_i : i\in\omega\}$ for $\Lat$.

\begin{defn}[Uniform tree]
	A \textbf{uniform tree} for the representation $\{\Theta_i : i\in\omega\}$ is a function $T$ with both domain and range the set of all strings $\sigma$ such that if $\sigma(n)$ is defined, then $\sigma(n)\in \Theta_n$, which
	satisfies the following properties for all $\sigma,\tau\in\dom{T}$:
	\begin{enumerate}
		\item (Order) If $\sigma\subseteq \tau$, then $T(\sigma)\subseteq T(\tau)$.
		\item (Nonorder) If $\sigma | \tau$, then $T(\sigma) | T(\tau)$. In fact, we require that for each length $l$ there is a string $\pi$ such that if $|\sigma| = l$ and $\alpha\in\Theta_{l}$ (so that $\sigma\concat\alpha\in\dom{T}$), then $T(\sigma\concat\alpha)\supseteq T(\sigma)\concat \pi \concat \alpha$.
		\item (Uniformity) For every fixed length $l$, there is a string $\pi_l$ and for every $\alpha\in \Theta_{l}$, there is a string $\rho_{l,\alpha}$ whose length does not depend on $\alpha$ such that if $|\sigma| = l$, then $T(\sigma\concat \alpha) = T(\alpha)\concat\pi_l\concat\rho_{l,\alpha}$. Note, we require that $\pi_l$ and $\rho_{l,\alpha}$ only depend on the length of $\sigma$.
	\end{enumerate}
	
	We say $T$ is \textbf{branch-coding-free} if for every length $l$ and every $\alpha \in \Theta_l$, the string $\pi_l\concat\rho_{l,\alpha}$, which is the extension corresponding to $\alpha$ at this level, does not do \textbf{unnecessary coding}, i.e., for each $j$ if $(\pi_l\concat\rho_{l,\alpha})(j)\in C$ then $j = |\pi_l|$ (and so $\alpha \in C$). (This means that the only time a member $\alpha$ of our coding set $C$ appears on a branch is when we are at a fork and we took the path corresponding to $\alpha$).
	
	The tree $T$ is \textbf{congruence-respecting} if for every length $l$, every $\alpha,\beta\in \Theta_l$, and every $x\in\Lat$ if $\alpha\equiv_{x}\beta$ then $\pi_l\concat \rho_{l,\alpha} \equiv_x \pi_l \concat \rho_{l,\beta}$ (or, equivalently, $\rho_{l,\alpha} \equiv_x \rho_{l,\beta}$).
\end{defn}

Our trees are related to those in Definition 2.4 of Kjos-Hanssen and Shore \cite{KjoshanssenShore2010}. We are afforded some simplifications, again, because our lattice is finite; we can preserve all congruences all the time, and have the same domain for all our trees. Our nonorder (and, consequently, uniformity) property are different in that we have the string $\pi$ which every branch at a level has to follow before splitting. This is a technical requirement that allows us to prove the fusion lemma (Kjos-Hanssen and Shore also need this modification, their fusion lemma, as written, does not produce a forcing condition). The requirement that the trees are branch-coding-free is new, and allows us to code.

\begin{nota}
		Uniform trees will be denoted by uppercase Roman letters, most frequently $T,S,R$, and we denote the set of branches of $T$ by $[T]$. Strings of members of the lattice representation will be denoted $\sigma,\tau,\rho,\nu$ and so on, we reserve $\pi$ for the $\pi$ in the uniformity property. We write the concatenation of $\sigma$ by $\tau$ as $\sigma\concat \tau$ and we confuse a string of length one with its value. So, for instance, we may write $\sigma\concat\alpha$ for $\sigma\in\prod_{i=0}^n \Theta_i$ and $\alpha\in\Theta_{n+1}$. 
	
		For technical reasons, we define the \textbf{height} of a level $l$ to be $|T(\sigma)\concat \pi_l|$ where $\sigma$ is any string of length $l$, and $\pi_l$ is the string as in the definition of the uniformity property. By uniformity, this is independent of the choice of $\sigma$ and so is well defined.
\end{nota}

Hyperarithmetic, branch-coding-free, congruence-respecting, uniform trees will be the conditions of our notion of forcing for producing almost-initial-segments. Observe that the identity tree satisfies all these properties.

\begin{defn}[Subtree]
	A uniform tree $S$ is a \textbf{subtree} of a uniform tree $T$ ($S\subseteq T$) if the range of $S$ is contained in the range of $T$.
\end{defn}

We single out two operations on uniform trees.

\begin{defn}
	If $T$ is a uniform tree and $\sigma\in \prod_{i=0}^l \Theta_i$ for some $l$ then we define $T_\sigma$  by:
	\[
	T_\sigma(\tau) = T(\sigma\concat \tau).
	\]	
	If $\mu \in \prod_{i=0}^l \Theta_i$ for some $l$ and $l\leq |T(\emptyset)|-1$ then  the \textbf{transfer tree of $T$ over $\mu$} ($T^\mu$) is the tree such that $T^\mu(\sigma)$ is the string $T(\sigma)$ but with its initial segment of length $l$ is replaced by $\mu$. (i.e., you change the root of $T$ by replacing the initial-segment of the right length by $\mu$). We write $T_\sigma^\mu$ for $(T_\sigma)^\mu$.
\end{defn}

\begin{prop}
	Let $T$ be a uniform tree. Then $T_\sigma$ and $T^\mu$ are uniform trees whenever they are defined. Furthermore, if $T$ is branch-coding-free, congruence-respecting, or hyperarithmetic, then $T_\sigma$ and $T^\mu$ are, correspondingly, branch-coding-free, congruence-respecting, or hyperarithmetic. Finally, $T_\sigma$ is a subtree of $T$ whenever it is defined, and $(T_\sigma)_\tau = T_{\sigma\concat\tau}$.
\end{prop}

\begin{defn}[Perfect forcing]
	 \textbf{$\{\Theta_i:i\in\omega\}$-perfect forcing} is the set $\P_{\{\Theta_i:i\in\omega\}} = \P$ of all hyperarithmetic, branch-coding-free, congruence-respecting, uniform trees ordered by the subtree relation, i.e., for $T,S\in\P$ we say $T$ extends $S$ ($T$ refines $S$), $T\leq_\P S$, if $T$ is a subtree of $S$.
\end{defn}

Now we have our notion of forcing, we can discuss the objects its conditions approximate. Clearly, for each length $l$ the set $\{T\in \P : |T(\emptyset)|>l \}$
is dense in $\P$. Consequently, a descending sequence of conditions $\{T_i\}_{i=0}^\infty$ meeting these dense sets will correspond to an object $\G$, an element of the product $\prod_{i=0}^\infty \Theta_i$, defined by $\G(n) = \alpha$ iff there is an $i$ such that $T_i(\emptyset)(n)\downarrow = \alpha$. For each $x\in \Lat$ we define
	\[
		\G^x(n) = \G(n)(x),
	\]
an element of $\omega^\omega$. Our embedding will take $x\in\Lat$ to the degree of $\G^x$. To ensure that $\G^\top\geq_h \O$ we need to do some coding.

\begin{defn}[Coding]
	Fix $x,y\in \Lat$ which join up nontrivially to $\top$ (i.e., $x \join y = \top$ and $x,y\neq \top)$. The \textbf{root-coding of $T$ for $\ang{x,y}$} is the number of occurrences of $g(x,y,0)$ and $g(x,y,1)$ in the string $T(\emptyset)$. A subtree $S$ of $T$ does \textbf{no more root-coding than $T$ for $\ang{x,y}$} if the root-coding of $T$ for $\ang{x,y}$ is the same as that for $S$ and it does \textbf{no more root-coding} if it does no more root-coding for each pair $\ang{x,y}$ which join up nontrivially to $\top$.
\end{defn}

Given $\G^x\oplus \G^y$ our decoding procedure is as follows:
	\begin{quote}
		On input $n$ search for the $n$th number $m$ such that $\G(m) \equiv_x g(x,y,0)$ and either $\G(m) \equiv_y g(x,y,0)$ or $\G(m) \equiv_y g(x,y,1)$. If $\G(m) \equiv_y g(x,y,0)$ we say $n$ is not in the set, and if $\G(m) \equiv_y g(x,y,1)$ then we say $n$ is in the set.
	\end{quote}
(Note that $g$ is the function in the definition of coding set, and that the decoding procedure for $\G^x \oplus \G^y$ depends on the order of $x$ and $y$.)

This procedure is clearly recursive in $\G^x\oplus \G^y$ as $\{\Theta_i:i\in\omega\}$ is recursive and to determine whether $\G(m)\equiv_z \alpha$ it suffices to know $\G^z(m)$. Notice that if $\alpha \equiv_x g(x,y,0)$ (and so, automatically, is equivalent mod $x$ to $g(x,y,1)$) and $\alpha\equiv_y g(x,y,0)$, then $\alpha \equiv_\top g(x,y,0)$, by the join property of USL tables, and so $\alpha = g(x,y,0)$ by the order properties of USL tables, and similarly for $g(x,y,1)$. 

Hence, to ensure that our decoding procedure gives the characteristic function of some set $X$, it suffices to construct a $\G$ such that, for each $n$, at the $n$th place where $G(m)= g(x,y,0)$ or $G(m) = g(x,y,1)$, then at that place it is $g(x,y,0)$ iff $n\notin X$ and is $g(x,y,1)$ iff $n\in X$. Therefore, a running theme for the remainder of this section will be meeting dense sets of various kinds without increasing the root-coding of a condition.

\subsection{The forcing relation}

With this strategy in mind we must define our language of forcing and our forcing relation, and show we can provide a sufficient degree of genericity without interfering with the coding procedure.

\begin{defn}[Languages and models]
	For each $x\in \Lat\setminus\{\top\}$ we define a language $\Lang[x]$ and a model $\Model[x]$ as described in Sacks chapter III, section 4 \cite{Sacks1990}. 
	
	Briefly, $\Lang[x]$ is second-order arithmetic augmented with ranked set variables $X^\delta$ for $\delta<\omegaoneck$ over which we can quantify in the usual manner. A minor change is that our generic object will be an element of $\omega^\omega$ rather than a subset of $\omega$. This changes the atomic formulas slightly, and requires a function symbol $\texttt{G}^{\texttt{x}}$ rather than a predicate.
	
	A formula is \textbf{ranked} if each of its set variables is ranked; it is $\Sigma_1^1$ if it has an initial block of (unranked) set existentials, then a ranked formula. For a sentence $(\exists X^\delta)\varphi(X^\delta)$ and a formula $\mathcal{H}(n)$ of rank at most $\delta$ that has only one free variable, $n$, which is a natural number variable, $\varphi(\hat{n}\mathcal{H}(n))$ is obtained by replacing each occurrence of $t\in X^\delta$ with $\mathcal{H}(t)$, for each first-order term $t$. This operation decreases full ordinal rank.
	
	$\Model[x]$, is the class of all reals definable from $\G^x$ by a formula of $\Lang[x]$. Sacks provides simultaneous inductive definitions of the interpretation of the formulas of $\Lang[x]$ and the class of reals $\Model[x]$. Provided that $\omega_1^{\G^x} = \omegaoneck$, $\Model[x]$ is precisely the class of reals which are hyperarithmetic in $\G^x$, and so our forcing language has a term for every real hyperarithmetic in $\G^x$ assuming that we can preserve $\omegaoneck$.
\end{defn}

\begin{defn}[Forcing relation]
	Let $\varphi$ be a sentence of $\Lang[x]$ and $T$ a forcing condition. We define $T\forces_x \varphi$ by induction
	\begin{enumerate}
		\item If $\varphi$ is ranked, then $T\forces_x \phi$ iff for every $\G \in [T]$, $\Model[{x}]\models \varphi$.
		\item If $\varphi$ is unranked and $\varphi = (\exists n)\psi(n)$, then $T\forces \varphi$ iff there is an $n\in\omega$ such that $T\forces\psi(\underline{n})$.
		\item If $\varphi$ is unranked and $\varphi = (\exists X^\delta)\psi(X^\delta)$ then $T \forces_x \varphi$ iff there is a term $\mathcal{H}(n)$ of rank at most $\delta$ such that $T\forces_x \varphi(\hat{n}\mathcal{H}(n))$.
		\item If $\varphi$ is unranked and $\varphi = (\exists X)\psi(X)$ then $T\forces_x \varphi$ iff there is a $\delta<\omegaoneck$ such that $T\forces_x (\exists X^\delta)\psi(X^\delta)$.
		\item If $\varphi$ is unranked and $\varphi = \psi_1 \land \psi_2$ then $T \forces_x \varphi$ iff $T\forces_x \psi_1$ and $T\forces_x \psi_2$.
		\item If $\varphi$ is unranked and $\varphi = \neg\psi$ then $T\forces_x\varphi$ iff for all $S \in \P$ extending $T$, $\neg S\forces_x \psi$.
	\end{enumerate}
\end{defn}

\begin{nota}
	We denote formulas of our forcing languages as $\phi,\psi$, and occasionally use $\mathcal{H}(n)$ for a ranked formula with only one free variable, $n$, which is a natural number variable.
	
	As $\forces_x$ only holds between forcing conditions and sentences in $\Lang[x]$, we shall omit the $x$ from the $\forces$ if we have declared from which language the sentences come.
\end{nota}

It is standard to define forcing to be equal to truth for atomic formulas of a forcing language. However, we treat all ranked formulas at this ground level and define forcing to be equal to truth for all of them. This obscures the fact, which we will need to establish, that given a condition $T$ and a sentence $\phi$ there is an extension of $T$ deciding $\phi$. 

\begin{defn}[Generic sequence]
	A sequence $\{T_i : i\in\omega\}$ of elements of $\P$ is $\Lang[x]$-\textbf{generic} if $T_{i+1}\leq_\P T_i$ for each $i$, and for every $\phi\in\Lang[x]$ there is an $i$ such that $T_i$ forces  either $\phi$ or $\neg\phi$. The sequence is \textbf{generic} if it is generic for each $x\in \Lat\setminus\{\top\}$.
\end{defn}

Observe that if $\{T_i : i\in\omega\}$ is generic, then it meets each dense set
\[
D_n = \{T\in \P : T(\emptyset)(n)\downarrow\},
\] 
and so there is a unique $\G\in\bigcap_{i\in\omega}[T_i]$. We call such a $\G$ the \textbf{generic}.

\begin{lem}[Standard lemmas]
	Let $T$ be a condition, $x\in\Lat\setminus\{\top\}$ and $\phi\in\Lang[x]$. Then the following hold:
	\begin{enumerate}
		\item(Consistency) $T\not\forces \phi \land \neg\phi$. 
		\item(Extension) If $S$ extends $T$ and $T\forces \phi$, then $S\forces \phi$. 
		\item (Density) There is an $S$ extending $T$ deciding $\phi$ (i.e., $S$ either forces $\phi$ or forces its negation).
		\item (Forcing and truth) If $\{T_i:i\in\omega\}$ is $x$-generic and $\G$ is the generic object, then $\Model[x]\models \phi$ iff there is an $i$ such that $T_i\forces \phi$.
	\end{enumerate}
\end{lem}
\begin{proof}[Proof of consistency and extension properties]
	First the consistency property. Suppose $\phi$ is ranked. Then it follows from the fact that $\Model[x]\not\models \phi\land \neg\phi$ that $T$ does not force both $\phi$ and its negation. If $\phi$ is unranked, then the definition of forcing for negation implies that $T$ does not force both $\phi$ and its negation.
	
	For the extension property, if $\phi$ is ranked and $S\leq_\P T$, then $[S]\subseteq [T]$, consequently, if every branch of $T$ satisfies $\phi$, then every branch of $S$ does too and so $S\forces \phi$. If $\phi$ is unranked, we proceed by induction on the full ordinal rank of $\phi$. Details can be found in Chapter IV Section 4 of Sacks \cite{Sacks1990}.
\end{proof}
The inductive step in a standard proof of the density property goes through as normal (using the definition of forcing for negation). The issue is with the base case: It is not clear that given a condition that there is a refinement such that every branch satisfies a particular ranked formula. To show the existence of such an $S$ we need to establish the, so called, fusion property of trees. We also have the concern of coding unnecessarily at the root of $S$.

Before establishing the fusion lemma, we need some technical facts about the forcing relation.

\begin{lem}
	The relation $T\forces \varphi$ restricted to $\Sigma_1^1$ sentences $\varphi\in\Lang[x]$ is $\Pi_1^1$.
\end{lem}
\begin{proof}
	Our situation is only slightly different from that in Sacks Chapter IV, Lemma 4.2. We have a slightly different notion of tree and of extension, but they are all uniformly arithmetic in codes for the trees, and so the complexity has not increased.
\end{proof}

\begin{defn}
	Let $\sigma \in \prod_{i=0}^l \Theta_i$ be a string and $x\in \Lat\setminus\{\top\}$. The \textbf{$x$-safe version of $\sigma$}, $\sigma_x$, is defined by taking each $n<|\sigma|$ such that $\sigma(n)=\alpha \in C$ and defining $\sigma_x(n) = \beta$ where $\beta\in \Theta_0\setminus C$ is the member of the lattice table which agrees with $\alpha$ modulo $x$, but is not coding, and otherwise not changing $\sigma$. (There may be more than once such $\beta$, so for each $x\in\Lat\setminus\{\top\}$ and coding $\alpha$ pick a particular $\beta$ and always use that.) Observe that $\sigma\equiv_x \sigma_x$.
	
	If $T$ is a condition and $S$ extends $T$ then the \textbf{$x$-safe version of $S$ with respect to $T$} is the condition $S^{T(\sigma_x)}$ where $\sigma$ is the string such that $T(\sigma) = S(\emptyset)$ and $\sigma_x$ is its $x$-safe version. Note that $\{\G^x : \G\in [S]\} = \{\G^x : \G\in [S^{T(\sigma_x)}]\}$. (As $T$ preserves congruences, and so $T(\sigma) \equiv_x T(\sigma_x)$, and the $x$-safe version of $S$ with respect to $T$ does no more root-coding than $T$. As $T(\sigma_x) = S^{T(\sigma_x)}(\emptyset)$, $\sigma_x$ does no coding, and $T$ is branch-coding-free.)
\end{defn}

\begin{lem}
	Let $S$ and $S'$ be conditions that have the same branches modulo $x$ (i.e., $\{\G^x : \G\in S\} = \{\G^x : \G\in S'\}$), and let $\phi$ be a sentence of $\Lang[x]$. Then $S\forces \phi$ iff $S'\forces \phi$. Hence, in particular, if there is an extension $S$ of $T$ forcing a sentence, then there is an extension of $T$ forcing that sentence which does no more root-coding than $T$: namely, the $x$-safe version of $S$.\label{samebrancheslemma}
\end{lem}

\begin{proof}
	\emph{Base case, $\phi$ is ranked:} By the definition of forcing for ranked formulas, as the branches of $S$ and the branches of $S'$ are the same modulo $x$, then every branch of $S$ satisfies $\phi$ iff every branch of $S'$ does.
	
	\emph{Inductive step:} If $\phi = \psi_1 \land \psi_2$, then $S\forces \phi$ iff $S\forces\psi_1$ and $S\forces\psi_2$ iff (by induction) $S'\forces \psi_1$ and $S'\forces \psi_2$ iff $S'\forces \psi_i \land \psi_2$.
	
	If $\phi = (\exists x)\psi(x)$ then $S\forces \phi$ iff there is an $n$ such that $S\forces \psi(\underline{n})$ iff there is an $n$ such that $S' \forces \psi(\underline{n})$ iff $S'\forces (\exists x)\psi(x)$.
	
	For $\phi = (\exists X^\delta)\psi(X^\delta)$ then the proof is similar to the natural number existential, but with a witnessing formula $\mathcal{H}$ of rank at most $\delta$ in place of $n$. 
	
	If $\phi = (\exists X)\psi(X)$ then $S\forces \phi$ iff there is a $\delta<\omegaoneck$ such that $S\forces (\exists X^\delta)\psi(X^\delta)$ iff there is a $\delta<\omegaoneck$ such that $S'\forces (\exists X^\delta)\psi(X^\delta)$ iff $S'\forces \phi$.
	
	The case when $\phi = \neg \psi$ is somewhat trickier, we need to be able to transform extensions $R$ of $S$ which force $\psi$ into extensions $R'$ of $S'$ also forcing $\psi$. By the inductive hypothesis, it suffices to ensure that the branches of $R'$ are the same as those of $R$ modulo $x$.
	
	Given $R\leq_\P S$ we define $R'(\sigma) = S'(\tau_\sigma)$ where $\tau_\sigma$ is the (unique) string such that $R(\sigma) = S(\tau_\sigma)$. $R'$ is hyperarithmetic as $R,S$ and $S'$ are (so we can use $R$ and $S$ to find $\tau_\sigma$ for any $\sigma$, and then plug this into $S'$), and, furthermore,  $R'$ is a branch-coding-free, congruence-respecting, uniform subtree of $S'$, as $R$ is for $S$. It remains to show that the branches of $R$ and $R'$ are the same modulo $x$.
	
	\textbf{Claim:} For each level $l$ the height of $S$ and of $S'$ are the same.
	
	The height of the $0$th levels of $S$ and $S'$ are, respectively $|S(\emptyset)\concat\pi^S_0|$ and $|S'(\emptyset)\concat \pi_0^{S'}|$. Every branch on $S$ extends $S(\emptyset)\concat\pi^S_0$ and so every branch on $S'$ must extend this string modulo $x$. By the implementation of the nonorder property, the different branches of $S$ at this level all disagree at the $|S(\emptyset)\concat \pi_0^{S}|$th place. If $x= \bot$ then there is only one branch on $S,S',R,R'$: The constant $0$ branch, and so the claim is shown, otherwise $x\neq \bot$ and there are $\alpha,\beta$ which disagree modulo $x$, and so there are branches of $S$ which disagree at the $|S(\emptyset)\concat \pi_0^{S}|$ place, modulo $x$.
	
	This must be reflected in $S'$, and so there is splitting at the $|S(\emptyset)\concat \pi_0^{S'}|$th place on $S'$ therefore, the height of $S'$ at level $0$ must be less than or equal to that of $S$. Interchanging $S$ and $S'$ in this argument shows the heights at the root must be equal.
	
	Now we proceed inductively: For $\sigma$ of length $l$ we assume $|S(\sigma)\concat \pi^S_l| = |S'(\sigma)\concat \pi^{S'}_l|$, hence, it suffices to show that $|\rho^{S\smallfrown}_{l,\alpha} \pi^S_{l+1}| = |\rho^{S'\smallfrown}_{l,\alpha}\pi^{S'}_{l+1}|$. But the argument is similar to the base case: We know there are mod $x$ disagreements in branches of $S$ at $|S(\sigma)\concat \pi_l^S {}\concat\rho_{\alpha,l}^S{}\concat \pi_l^S|$, consequently, there must be mod $x$ disagreements in branches of $S'$ at this place too. Hence, the height of the $l+1$ level of $S'$ is at most the height of the $l+1$st level of $S$. Interchanging $S$ and $S'$ shows they must be equal. 
	
	\textbf{Claim:} For every $\sigma$ and every level $l$,  $S(\sigma)\concat\pi^S_l \equiv_x S'(\sigma)\concat\pi^{S'}_l$. 
	
	Every branch of $S$ extends $S(\emptyset)\concat \pi^S_0$ and so they all agree modulo $x$. This is reflected in the branches of $S'$ and so $S'(\emptyset)\concat\pi^{S'}_0$ must agree with the initial segment modulo $x$.
	
	Then, inductively, if it is true up to level $l$ then consider $S(\sigma)\concat\pi^S_l{}\concat\rho^S_{l,\alpha}{}\concat \pi^S_{l+1}$ for any $\alpha\in\Theta_l$. By induction $S(\sigma)\concat\pi^S_l$ and $S'(\sigma)\concat\pi^{S'}_l$ agree modulo $x$ and by the last claim they are of the same height. Then, by the implementation of nonorder, the first place where this string is undefined but $S(\sigma)\concat\pi^S_l{}\concat\rho^S_{l,\alpha}{}\concat \pi^S_{l+1}$ is takes value $\alpha$. As $S$ is congruence-respecting, then if $\alpha\equiv_x \beta$ then $\rho^S_{l,\beta}\equiv_x \rho^S_{l,\alpha}$, and so every branch of $S$ (and so of $S'$) which looks like $\alpha$ (modulo $x$) at this place looks the same for the rest of the string $\rho^S_{l,\alpha}$ modulo $x$. Consequently, $\rho^{S'}_{l,\alpha}$ must agree modulo $x$ with $\rho^{S}_{l,\alpha}$ because $S'$ is also congruence-respecting.
	
	\textbf{Claim:} The $R$ and $R'$ above have the same branches modulo $x$.
	
	If $G\in[R]$ is a path, there is some sequence $\{\sigma_i:i\in\omega\}$ of compatible strings such that $R(\sigma_i)$ converges to $G$, and $|\sigma_i|=i$. As $R$ is a subtree of $S$ there is a sequence of compatible strings $\{\tau_i:i\in\omega\}$ such that $S(\tau_i)= R(\sigma_i)$. By the previous claim, $S'(\tau_i) \equiv_x S(\tau_i)$, and so 
	\[
		R'(\sigma_i) = S'(\tau_i) \equiv_x S(\tau_i) = R(\sigma_i)
	\] 
	and so there is some $G'\in [R']$ such that $G\equiv_x G'$. Interchanging the role of $R$ and $R'$ shows that every branch of $R'$ has a corresponding branch in $R$ which agrees modulo $x$, hence the branches are the same modulo $x$, as required.
	
	So, suppose $S'\forces \neg\psi$, then there is no $R'\leq S'$ such that $R'\forces \psi$. If there were $R\leq S$ forcing $\psi$ then, we can construct $R'\leq S'$ which forces $\psi$ (as we can construct $R'$ with the same branches as $R$ modulo $x$ and so, by induction, forcing $\psi$). Consequently, $S$ must force $\neg\psi$, which completes the proof.
\end{proof}

\begin{lem}[Fusion]
	Let $\{\phi_i:i\in\omega\}$ be a hyperarithmetic sequence of $\Sigma_1^1$ formulas of $\Lang[x]$, and let $T$ be a condition such that for every $S\leq_\P T$ and every $j\in\omega$, there is an $R\leq_\P S$ such that $R\forces \phi_j$. Then there is a condition $V\leq_\P T$ forcing each $\phi_i$ which does no more root coding than $T$.
\end{lem}
\begin{proof}
	Fix such a sequence $\{\phi_i:i\in\omega\}$ and a condition $T$. By the previous lemma, for every $S\leq_\P T$ and $i\in\omega$, as there is an $R\leq_\P S$ forcing $\phi$, there is one forcing $\phi$ and doing no more root-coding than $S$. Consider the predicate 
	\begin{quote}
		$R\in\P$, refines $S\in\P$, does no more root-coding than $S$, and forces $\phi_i$.
	\end{quote}
	As $\P$ and the forcing relation are $\Pi_1^1$ and the other clauses are arithmetic, this predicate is uniformly $\Pi_1^1$ in $R,S,i$. By Krisel's uniformization theorem, there is a partial $\Pi_1^1$ function which produces $R$ in terms of $S$ and $i$. We denote this function $R(S,i)$. By assumption, $R$ is total on the conditions $S$ extending $T$.
	
	We want to construct a single condition $V$ that extends $T$ and forces each $\phi_i$ simultaneously. We construct $V$ level by level, as well as auxiliary conditions $U^{l}_j$ where $l\in\omega$ is a level, and $j$ varies between $0$ and the number $m(l)$, which is one less than the number of branches of $T$ at level $l$ (i.e., $m(l) = \prod_{k=0}^l |\Theta_k|-1$). Fix a simultaneous hyperarithmetic enumeration $\alpha^l_{k}$ of each $\Theta_l$, and order strings lexicographically. 
	
	\textbf{Stage $0$:} $V(\emptyset) = T(\emptyset)$.
	
	We define $U^0_0$ to be $R(T_{\alpha^0_0},0)$, a subtree of $T_{\alpha^0_0}$ which forces $\phi_0$, and does no more root-coding than $T_{\alpha^0_0}$. By the uniformity of $T$, $(U^0_0)^{T(\alpha^0_1)}$ is a subtree of $T_{\alpha^0_1}$ and so of $T$. We define $U^0_1 = R((U^0_0)^{T(\alpha^0_1)},0)$, we continue in this fashion across the level defining $U^0_{k+1} = R((U^0_{k})^{T(\alpha^0_k)},0)$. At the end we have $U^0_{m(0)}$. By uniformity $(U^0_{m(0)})^{T(\alpha^0_k)}$ is a subtree of $U^0_k$ for every $k$, and as such, must force $\phi_0$.
	
	We define $V(\alpha) = (U^0_{m(0)})^{T(\alpha)}(\emptyset)$ for each $\alpha \in \Theta_0$. By the uniformity of both $T$ and $U^0_{m(0)}$, $V$, as defined so far, is uniform and congruence-respecting. 
	
	To see it is branch-coding-free, observe that if $V(\alpha)(n) = \beta \in C$ and $n\geq |V(\emptyset)|$, then, by definition, $(U^0_{m(0)})^{T(\alpha)}(\emptyset)(n)=\beta$. If $n<|T(\alpha)|$, then $(U^0_{m(0)})^{T(\alpha)}(\emptyset)(n)=T(\alpha)(n)$, and as $T$ is branch-coding-free, this means that  $\alpha = \beta$ and $n$ is precisely $|T(\emptyset)\concat \pi^T_0|$, which is not unnecessary coding. Otherwise, $n\geq|T(\alpha)|$. Let $k$ be the first stage such that $U^0_k(\emptyset)(n)\downarrow$. $U^0_0$ does no more root-coding than $T(\alpha^0_0)$, by the choice of $R$, and so, inductively across the level, $U^0_{j+1}$ does no more root-coding than $(U^0_{j})^{T(\alpha^0_{j+1})}$ which does no more root-coding than $T_{\alpha^0_{j+1}}$. Therefore, $\beta = \alpha^0_{k}$ and $n$ must be precisely $|T(\alpha^0_k)\concat \pi^T_1|$ which is not unnecessary coding.

	\textbf{Stage $l>0$:} Assume we have defined $V$ up to level $l$, and so far it is branch-coding-free, congruence-respecting, and uniform, and that we have $U^{l-1}_{m(l-1)}$ a tree which, by induction, is a forcing condition which is a subtree of $T_{\tau}$ where $\tau$ is last string in our uniform enumeration of strings of length $l$ and has root at least as long as the height of $V$ so far.
		
	Now we define the $U^{l}_k$ for $k\in\{0,\ldots, m(l+1)\}$. Starting with the least string $\sigma\concat \alpha$ of length $l$ we set $U^{l}_{0} = R((U^{l-1}_{m(l-1)})^{V(\sigma\concat\alpha)}_{0^{l-1}{}\concat\alpha},l)$, then given $U^{l}_k$ and the $k+1$th string $\sigma\concat\alpha$ we define
	\[
		U^{l}_{k+1} = R((U^{l}_{k+1})^{S(\sigma\concat\alpha)},l).
	\]
	At the end we have $U^l_{m(l)}$, and by a similar argument as in the base case, if $\sigma\concat\alpha$ is the $k$th string of length $l$ then $(U^l_{m(l)})^{U^l_{k}(\emptyset)}$ is a subtree of $U^l_{k}$ which forces $\phi_l$. So, we define
		\[
			V(\sigma\concat\alpha) = (U^l_{m(l)})^{U^l_{k}(\emptyset)}(\emptyset).
		\]
	This preserves that $V$ is (so far) branch-coding-free congruence-preserving and uniform, as each $U^l$ is. We also have that the root of $U^l_{m(l)}$ is sufficiently long and has the various other properties assumed by the induction.
	
	This completes the inductive construction of $V$, which is a forcing condition extending $T$ doing no more root-coding. $V$ forces each $\phi_i$ as for every string $\sigma$ of length $i$ every path on $V$ extending $V(\sigma)$ is a path on one of the $U^i_k$s, and so by construction of the $U$s and the fact that the formulas are $\Sigma_1^1$ every path on $V$ makes $\phi_i$ true, and so $\phi_i$ is forced by $V$.
\end{proof}

With the fusion lemma in hand, we can complete the proofs of the standard lemmas regarding the forcing relation.

\begin{proof}[Proof of the density property]
	Let $\phi$ be a sentence in $\Lang[x]$ and $T$ a forcing condition. If $\phi$ is unranked, then there is an $S\leq_\P T$ deciding $\phi$, by the definition for forcing the negation of an unranked formula. By choosing the $x$-safe version of $S$, we can also ensure that $S$ does no more root-coding than $T$.
	
	If $\phi$ is ranked, then we proceed by induction on the full ordinal rank and logical complexity of the formula. To decide atomic sentences, we can pick $T_\sigma$ for some sufficiently long $\sigma$. Note we can choose $\sigma$ so as to do no more root-coding.
	
	The induction for cases for $\land$ and $\neg$ are standard, the difficulty comes with existential quantifiers. For instance, if $\phi = (\exists X^\delta)\psi(X^\delta)$, then let $\mathcal{H}_i(n)$ be an effective enumeration of all formulas of rank at most $\delta$ whose sole free variable is $n$. If there is an $i$ and an $S\leq_\P T$ such that $S\forces \psi(\hat{n}\mathcal{H}_i(n))$, then $S\forces \psi$, and we could choose the $x$-safe version of $S$ so as to do no more root-coding.
	
	Otherwise, for each $S\leq_\P T$ and $i$, $S\not\forces \varphi(\hat{n}\mathcal{H}(n))$, and so, by induction, for each $i$ and $S\leq_\P T$ there is an $R\leq_\P S$ such that $R\forces \neg \psi(\hat{n}\mathcal{H}(n))$.
	
	Now we can apply the fusion lemma to the sequence $\neg\psi(\hat{n}\mathcal{H}_i(n))$ to find an $S\leq_\P T$ forcing each $\neg \psi(\hat{n}\mathcal{H}_i(n))$, and so $S\forces \neg\psi$. Furthermore, we can choose $S$ to do no more coding, as the fusion lemma allows this.
	
	The case for a natural number existential is similar. Observe that each extension could be chosen to do no more root-coding. 
\end{proof}

\begin{proof}[Proof that truth is forcing]
	Suppose $\{T_i:i\in\omega\}$ is $x$-generic and $\G$ is the generic object. Firstly, let $\phi$ be a ranked formula. Then, as the sequence is $x$-generic, there is an $i$ such that $T_i$ decides $\phi$. By the definition of forcing for ranked formulas, either every branch of $T_i$ satisfies $\phi$ or every branch satisfies its negation. In particular, as $\G\in[T_i]$, if $\G^x$ satisfies $\phi$ then every branch does, and if $\G^x$ satisfies its negation, then every branch does.
	
	If $\phi$ is unranked, then we proceed by induction on the logical complexity of $\phi$. The proof, from here, is standard.
\end{proof}

\begin{lem}
	Let $\G$ be $x$-generic. Then $\G^x$ preserves $\omegaoneck$.
\end{lem}
\begin{proof}
	As in Chapter IV Section 5 of Sacks \cite{Sacks1990}, the fusion lemma provides the proof.
\end{proof}

Observe that we can construct a generic sequence, by, step-by-step, extending the current condition to decide the next sentence and at no point do we have to increase the root-coding of the current condition. Now we turn to showing we can force our embedding to have the properties we need.

\subsection{Producing almost-initial-segments}

We want to verify that we can construct a generic sequence $\{T_i:i\in\omega\}$ such that the generic object $\G$ induces an embedding $x\mapsto \text{degree}(\G^x)$ which preserves $\sqleq$, $\not\sqleq$, $\join, \meet$, sends $\Lat\setminus\{\top\}$ to an initial segment, and sends $\top$ to $\O$.

\begin{nota}
	For each $x\in\Lat\setminus\{\top\}$, we number the ranked terms $\hat{n}\mathcal{H}$ of $\Lang[x]$ by ordinals $\delta<\omegaoneck$ and denote the characteristic function of the set they stand for by $\{\delta\}^{\G^x}$ (of course, $\{\delta\}^{\G^x}$ depends on a choice of forcing condition, and need not be total).
\end{nota}

\begin{defn}
	A condition $T$ \textbf{decides} $\{\delta\}^{\G^x}$ via $q$ (a map into $\{0,1\}$), if for every $n$ and $\sigma\in\dom{T}$ of length $n$, $T_\sigma\forces \{\delta\}^{\G^x}(n) = \underline{q(n,\sigma)}$.
\end{defn}

\begin{lem}
	Let $T$ be a condition which decides $\{\delta\}^{\G^x}$ via $q$. Then $q$ is hyperarithmetic.
\end{lem}

\begin{proof}
	As the forcing relation is uniformly $\Pi_1^1$ (as the formulas are ranked), $q$ is a total $\Pi_1^1$ function and; therefore, is hyperarithmetic.
\end{proof}

\begin{lem}
	Let $T$ be a condition, $x\in\Lat$, and $\delta<\omegaoneck$. Then there is an $S\leq_\P T$ which decides $\{\delta\}^{\G^x}$, and does not more root-coding than $T$
\end{lem}

\begin{proof}
	This follows from the coding-free fusion lemma.
\end{proof}

As our lattice representation $\{\Theta_i:i\in\omega\}$ is recursive and each $\Theta_i$ is a USL table, our map, $x\mapsto \text{degree}(\G^x)$, preserves $\sqleq$ and $\join$. We also need our map to be injective and to preserve $\not\sqleq$. Injectivity follows from preservation of $\not\sqleq$, so we concentrate on preserving $\not\sqleq$.

We want to show for each $x\not\sqleq y$ that $\G^x$ is not hyperarithmetic in $\G^{y}$. As $\G^{y}$ preserves $\omegaoneck$ for each $y\neq\top$ it suffices to show that $\G_x\notin\Model[y]$, i.e., that there is no term $\{\delta\}^{\G^{y}}$ in the language $\Lang[y]$ which defines $\G^x$. Clearly, we need not consider $y=\top$ as there is no corresponding $x$ not below $\top$.

\begin{lem}[Diagonalization]
	Let $x,y\in\Lat$ satisfy $x\not\sqleq_\Lat y$, let $\delta<\omegaoneck$, and let $T$ be a condition. Then there is an $n\in\omega$ and an extension of $T$ which does no more root-coding than $T$, decides the values of $\G^x(n)$ and of $\{\delta\}^{\G^{y}}(n)$, and decides them to be different.
\end{lem}

\begin{proof}
	Firstly, as $y$ is not above $x$, $y\neq \top$ and so $\Lang[y]$ and $\Model[y]$ are defined. We may assume $T$ decides $\{\delta\}^{\G^y}$ via $q$ (possibly by replacing $T$ with an extension doing no more root-coding), and we fix $\alpha,\beta\in\Theta_0\setminus C$ differentiating $x$ and $y$, i.e., $\alpha\equiv_y \beta$ yet $\alpha\not\equiv_x \beta$.
	
	By the uniformity of $T$ there is a string $\pi$ such that $T(\alpha)\supseteq T(\emptyset)\concat\pi\concat \alpha$ and $T(\beta)\supseteq T(\emptyset)\concat\pi\concat\beta$. In particular, if $n = |T(\emptyset)\concat\pi|$ then $T_\alpha(\emptyset)(n)= \alpha$ and $T_\beta(\emptyset)(n) = \beta$. Consequently, every branch of $T_\alpha$ disagrees with every branch of $T_\beta$ at $n$ modulo $x$.
	
	Let $\alpha^n$ be $n$ many copies of $\alpha$ and define $\beta^n$ similarly. As $T$ decides $\{\delta\}^{\G^y}$ via $q$ 
		\[
			T_{\alpha^n}\forces \{\delta\}^{\G^y}(n) = q(n,\alpha^n)\text{ and } T_{\beta^n}\forces \{\delta\}^{\G^y}(n) = q(n,\beta^n).
		\]
	I claim that, further, $q(n,\alpha^n)=q(n,\beta^n)$. To see this, note that as $\alpha\equiv_y \beta$ then $T_{\alpha^n}$ and $T_{\beta^n}$ have the same branches modulo $y$. This implies, by Lemma \ref{samebrancheslemma}, that $T_{\alpha^n}$ and $T_{\beta^n}$ force precisely the same sentences of $\Lang[y]$, which establishes the claim.
	
	In summary, $T_{\alpha^n}$ and $T_{\beta^n}$ force the same value of $\{\delta\}^{\G^y}(n)$ yet force different values of $\G^x(n)$. As such, at least one of $T_{\alpha^n}$ or $T_{\beta^n}$ diagonalizes against the $\delta$th reduction. Neither tree does more root-coding than $T$, as $T$ is branch coding free and neither $\alpha$ nor $\beta$ are in $C$, hence we have established the existence of the desired condition.
	\label{diagonalizinglemma}
\end{proof}

So, repeatedly applying the above Lemma, we can force our map to be an uppersemilattice embedding. We could force the preservation of meets by meeting appropriate dense sets, but we get it for free provided we can force the embedding to be an almost-initial-segment. To this end we need to establish the existence of splitting subtrees, which is considerably more complicated than anything we have done so far.

\begin{defn}
	For each reduction $\delta$ and condition $T$ deciding $\{\delta\}^{\G^x}$ via $q$, we say that $\sigma$ and $\tau$ (of the same length) are {\bf$(\delta,x)$-splitting on $T$ (modulo $y$)} if $(\sigma\equiv_y \tau$ and) there is an $n\leq|\sigma|$ such that $q(n,\sigma\rest n) \neq q(n,\tau\rest n)$.
\end{defn}

\begin{lem}
	Let $\delta$ be a reduction and $T$ a condition deciding $\{\delta\}^{\G^x}$ (where $x\neq \top$). There is a $\rho$ such that the set
		\[
			\text{Sp}(\rho) = \{y\in\Lat : \text{there are no $\sigma,\tau$ that $(\delta,x)$-split on $T_\rho$ modulo $y$}\}	
		\]
	is maximal. Moreover, this set is closed under meet, and so has a least element, and we can choose such a $\rho$ which does no coding.
\end{lem}

\begin{proof}
	As $\Lat$ is finite there is clearly a $\rho$ such that $\text{Sp}(\rho)$ is maximal. I claim that if we replace $\rho$ by its $x$-safe version then $\text{Sp}(\rho_x)$ is still maximal. To see this, it suffices to observe that, as $T_\rho$ and $T_{\rho_x}$ have the same branches modulo $x$, then they both decide $\{\delta\}^{\G^x}$ via the same map $q$, hence $\sigma,\tau$ $(\delta,x)$-split on $T_\rho$ iff they $(\delta,x)$-split on $T_{\rho_x}$. Thus, $\text{Sp}(\rho) = \text{Sp}(\rho_x)$, and one is maximal iff the other is.
	
	Now we need to show $\text{Sp}(\rho)$ is closed under meet. Suppose $y,z\in \text{Sp}(\rho)$, we want to show that there are no $(\delta,x)$-splits on $T_{\rho\concat 0}$ modulo $y\meet z$ (here we extend $\rho$ by one place for technical reasons). Suppose there was such a split $\sigma$ and $\tau$. By the existence of meet-interpolants there are $\hat{\gamma_1},\hat{\gamma_2},\hat{\gamma_3}\in\prod_{i=1}^{|\sigma|+1}\Theta_i$ such that for each $0<j<|\sigma|+1$ $\hat{\gamma_1}(j),\hat{\gamma_2}(j),\hat{\gamma_3}(j)$ are meet interpolants for $\sigma(j)$ and $\tau(j)$. In particular,
		\[
			\sigma \equiv_y \hat{\gamma_1}\equiv_z \hat{\gamma_2}\equiv_y \hat{\gamma_3} \equiv_z \tau.
		\]
	Now as $\sigma$ and $\tau$ form a $(\delta,x)$-split on $T_{\rho\concat 0}$ then so too do one of the consecutive pairs, listed above. But then, supposing it is the first pair, $0\concat\sigma$ and $0\concat\hat{\gamma_1}$ forms a $(\delta,x)$-split modulo $y$, contradicting the fact that $y\in\text{Sp}(\rho)$. The other pairs are similar, and so there are no $y\meet z$ splits on $T_\rho$ as required.
\end{proof}

Using the above lemma we construct the splitting subtrees:

\begin{lem}
	Let $\delta$ be a reduction and $T$ a condition deciding $\{\delta\}^{\G^x}$ (where $x\neq \top$) via $q$, let $\rho$ be a string such that $\text{Sp}(\rho)$ is maximal yet $\rho$ does no coding, and let $z$ be the least element of $\text{Sp}(\rho)$. Then there is a condition $S$ extending $T$ and doing no more root-coding than $T$ such that for any $\sigma,\tau$ if $\sigma\not\equiv_z\tau$ then $\sigma$ and $\tau$ $(\delta,x)$-split on $S$. We call such an $S$ a $z-(\delta,x)$-splitting tree.
\end{lem}

\begin{proof}
	We define $S$ inductively, level by level. We begin with $S(\emptyset) = T(\rho)$, which, by choice of $\rho$, does no more root-coding than $T$. Now suppose we have defined $S(\sigma)=T(\tau_\sigma)$ for each $\sigma$ of length $l$. We must define $S(\sigma\concat\alpha)$ for all such $\sigma$ and $\alpha\in\Theta_l$ in a  congruence-respecting, branch-coding-free, and uniform fashion, across this level.
	
	List the strings of length $l+1$ as $\sigma_j\concat\alpha_j$ for $j<m= \left| \prod_{i=0}^{l}\Theta_i\right|$. We define by a subinduction on $r< m(m-1)/2$ strings $\rho_{j,r}$ (simultaneously for $j<m$) and we will set
		\[
			\tau_{\sigma_j\concat\alpha_j} = \tau_{\sigma_j}\concat\alpha_j\concat\rho_{j,0}\concat\cdots\concat\rho_{j,m(m+1)/2}.
		\]
	We maintain uniformity by ensuring that $|\rho_{j,r}|=|\rho_{j',r}|$ for each $j,j'$,  we respect-congruences by insisting if $\alpha_j\equiv_y\alpha_{j'}$ then $\rho_{j,r}\equiv_y \rho_{j',r}$ for each $r$ and y, and we do no unnecessary coding by ensuring that each $\rho_{j,r}$ never takes on a coding value. Provided this is all effective, we will have a condition at the end.
	
	By induction on $r<m(m+1)/2$ suppose we have $\tau_{j}\concat\alpha_j\concat\rho_{j,0}\concat\cdots\concat\rho_{j,r-1} = \nu_j$ for all $j<m$. Suppose $\{p,q\}$ is the pair of distinct numbers both less than $m$ numbered by $r$. We wish to force a split corresponding to $\alpha_p$ and $\alpha_q$ if necessary. If $\alpha_p\equiv_z \alpha_q$ then we need not force a split, and so we can define $\rho_{j,r} = \emptyset$ for each $j<m$.
	
	Otherwise, let $y$ be the largest $w\in\Lat$ such that $\alpha_p\equiv_w\alpha_q$, of course $z\not\sqleq y$. By choice of $z$, there are $\sigma,\tau$ such that $\nu_p$ extended by $\sigma$ and $\tau$ form a $(\delta,x)$-splitting modulo $y$ on $T_\rho$. Consequently, $\nu_q\concat \tau$ must also $(\delta,x)$ split with one of $\nu_p\concat\alpha$ and $\nu_p\concat\tau$. If it splits with $\nu_p\concat\tau$, then we set $\rho_{j,r+1} = \tau_x$ the $x$-safe version of $\tau$. This is uniform and congruence-respecting (because we are picking the same extension for each $j$) and, furthermore, $\nu_p\concat\tau_x$ and $\nu_q\concat\tau_x$ still form a $(\delta,x)$-split, because $\nu_p\concat\tau_x\equiv_x \nu_p\concat \tau$ and so they force the same values for $\{\delta\}^{\G^x}$ wherever defined.
	
	Now suppose $\nu_q\concat\tau$ splits with $\nu_p\concat\sigma$. If $\alpha_p\equiv_w\alpha_q$ then $w\sqleq y$ by maximality of $y$, and so $\sigma\equiv_w \tau$, as $\sigma\equiv_y \tau$. We pick homogeneity interpolants $\gamma_0(s),\gamma_1(s)$ in $\Theta_{s+1}$ and $\Lat$-homomorphisms $f_s,g_s,h_s:\Theta_s\rightarrow\Theta_{s+1}$ such that
		\[
			f_s:\alpha_p,\alpha_q\mapsto \sigma(s),\gamma_1(s),\quad g_s:\alpha_p,\alpha_q\mapsto \gamma_0(s),\gamma_1(s),\quad h_s:\alpha_p,\alpha_q\mapsto \gamma_0(s),\tau(s).
		\]
	As $\nu_p\concat\sigma$ and $\nu_q\concat\tau$ $(\delta,x)$-split on $T_\rho$ one of the pairs $\nu_p\concat\sigma,\nu_q\concat\hat{\gamma_1}$, or $\nu_p\concat\hat{\gamma_0},\nu_q\concat\hat{\gamma_1}$, or $\nu_p\concat\hat{\gamma_0},\nu_q\concat\tau$ must $(\delta,x)$-split on $T_\rho$ too. We set $\rho_{j,r+1}(s) = f_s(\alpha_j)$ or $g_s(\alpha_j)$ or $h_s(\alpha_j)$ corresponding to which pair splits. This is uniform as $\rho_{j,r+1}$ depends only on $\alpha_j$, and as each $f_s,g_s,h_s$ are $\Lat$-homomorphisms we respect-congruences.
	
	The final thing to show is that we have not done unnecessary coding. Well, as $\{\Theta_i:i\in\omega\}$ is coding-ready, then it is acceptable for $A$, so by definition $f_s(\alpha)\in C$ implies that $\alpha = \sigma(s)$ or $\alpha = \tau(s)$ (and the same for $g_s,h_s$). So, it suffices to show that we can pick $\sigma$ and $\tau$ which themselves do no coding.
	
	As $\nu_p = \tau_{\sigma_j}\concat \alpha_j\concat\rho_{j,0}\concat\cdots\concat\rho_{j,r-1}$ then $|\nu_p|>0$, therefore the $\sigma,\tau$ we are trying to pick live in $\prod_{i=k}^{i=k'}\Theta_i$ for some $k'>k>0$. Consequently, if $\sigma(s)\in C$ then $\sigma(s)\in\Theta_0 \subseteq \Theta_{k+s}^\ast$ (similarly for $\tau(s)$) and so by the definition of acceptable for $A$, there are $\sigma'(s)$ and $\tau'(s)\in\Theta_{k+s}\setminus\Theta_{k+s}^\ast$ such that
		\[
			\sigma(s)\equiv_{x'}\sigma'(s),\tau(s)\equiv_{x'}\tau'(s)\text{, and for all $w\in\Lat$}[\sigma(s)\equiv_w\tau_s\Rightarrow \sigma'(s)\equiv_w\tau'(s)].
		\]
	If we pick a coatom $x'>x$ then we can construct $\sigma',\tau'$ which do not take coding values and such that $\sigma'\equiv_{x'}\sigma$ and $\tau'\equiv_{x'}\tau$, and, therefore, which still form a $(\delta,x)$-split on $T_\rho$ modulo $y$. Thus, when we picked $\sigma,\tau$ we could have picked them to do no coding, and then nothing else can code as $\{\Theta_i:i\in\omega\}$ is acceptable for $A$.
\end{proof}

\begin{lem}
	If $T$ is a $z-(\delta,x)$-splitting tree, then $T$ forces $\{\delta\}^{\G^x}\equiv_h \G^z$.
\end{lem}

\begin{proof}
	Fix $\G \in[T]$. We first show $\G^z\geq_h \{\delta\}^{\G^x}$. Pick an $n$, using $\G^z$ find all $\sigma\in\dom{T}$ of length $n$ such that $T(\sigma)(m)(z) \equiv_z\G^z(m)$ for all $m\leq n$, i.e, narrow down the possible paths through $T$ to those consistent with $\G^z$. All these $\sigma$ are equivalent modulo $z$ and so each $T_\sigma$ forces the same value of $\{\delta\}^{\G^x}(n)$, by the choice of $z$. As $T(\sigma)$ is an initial segment of $\G$ for one of these $\sigma$, then $\{\delta\}^{\G^x}(n)$ must be the correct value for that $\sigma$, and hence, all such $\sigma$.
	
	Now we compute the other way. Given $\{\delta\}^{\G^x}$ consider all $\sigma,\tau\in\dom{T}$ of length $n$. If $\sigma\not\equiv_z \tau$ then $\sigma$ and $\tau$ form a $(\delta,x)$-split on $T$ and so, in particular, $T_\sigma$ and $T_\tau$ force different values for $\{\delta\}^{\G^x}$ at some $m<n$. Thus, of the $\equiv_z$ equivalence classes of $\sigma$s and $\tau$s of length $n$, only the correct one will force the correct value of $\{\delta\}^{\G^x}$, so we can rule out all of the incorrect ones, as $T$ and $q$ are hyperarithmetic, leaving us with a single $\equiv_z$ equivalence class, these all determine the same initial segment of $\G$ modulo $z$, and so it must be correct modulo $z$.
\end{proof}

This is the penultimate step on our way to Theorem \ref{maintheorem1}: We know we can construct an embedding of $\Lat$ into $\Degh$ such that $\Lat\setminus\{\top\}$ is an initial segment. We do this by diagonalizing against each hyperarithmetic reduction $\delta$ for each pair $x\not\sqleq y\in \Lat$ and by constructing splitting trees. This is only countably many requirements so we can satisfy each in turn. We also need to decide each sentence of $\Lang[x]$ for each $x\in\Lat\setminus\{\top\}$ and to preserve $\omega_1^{CK}$ for each such $x$.

We have proved that we can do all this without ever doing coding at the root of a condition, consequently, we can intersperse these requirements with requirements saying: 

If $n$ is the first place we are yet to code for the pair $x,y$ joining up nontrivially to $\top$, then if $n\in\O$ take $T_{g(x,y,1)}$ and if $n\not\in\O$ take $T_{g(x,y,0)}$ as the next condition (where $g$ is the function in the definition of coding set). 

This implements our coding scheme and so $\G^\top\geq_h \O$. How do we guarantee that $\G^\top\leq_h \O$? It suffices to construct a generic sequence hyperarithmetically in $\O$. But this is only so much checking: The notion of forcing and the extension relation are hyperarithmetic in $\O$, as are the languages $\Lang[b]$. Also, the various constructions we effected can all be made uniformly hyperarithmetic in $\O$ by always picking "least" strings or extensions doing various things, with some fixed hyperarithmetic enumeration of strings.

The reader may also note that we could code in any set for $\top$, and provided that set $X$ is hyperarithmetically above $\O$ then there is a generic $\G$ with top element having the same hyperdegree as $X$.

\section{Extending embeddings}

In this section we want to establish Theorem \ref{maintheorem2}, which says that if $\U$ and $\V$ are USL$^\top$s and $\U$ is an almost-initial-segment of $\V$, then every embedding of $\U$ into $\Degh(\leq\O)$ extends to one of $\V$ (here, embedding means an injective map, preserving $\join$, mapping $\bot$ to the degree of the hyperarithmetic sets, and $\top$ to the degree of $\O$). Our strategy is to prove the result for two special cases and to show that, together, these imply the full result.

\begin{defn}[Free extension]
	For a USL$^\top$ $\U$, and a set $X$ disjoint from $\U$ the \textbf{$\top$-preserving free extension of $\U$ by $X$} ($\U[X]$) is the USL$^\top$ with domain
		\[
			V = ((\U\setminus\{\top\})\times [X]^{<\omega})\cup \{\top\},
		\]
	with the partial order $\sqleq_\V$ defined by setting $\top$ to be greater than everything, and declaring $(u_1,X_1)\sqleq_\V (u_2,X_2)$ for $u_1,u_2\in U$ and $X_1,X_2$ finite subsets of $X$ iff $u_1\sqleq_\U u_2$ and $X_1\subseteq X_2$. Its least element is $(\bot,\emptyset)$, its greatest element is $\top$, and $\join_\V$ is defined by
		\[
			(u_1,X_1)\join_\V (u_2,X_2) = (u_1\join_\U u_2, X_1\cup X_2)
		\]
	provided $u_1\join_\U u_2 \neq \top$ and is $\top$ otherwise.
\end{defn}

\begin{prop}
	$\V$ as defined above is a USL$^\top$, and there is a natural embedding from $\U$ to $\V$ taking $\top_\U$ to $\top_\V$ and taking any other $u\in\U\setminus\{\top\}$ to $(u,\emptyset)$, which realizes $\U$ as an almost-initial-segment of $\V$.
\end{prop}

\begin{proof}
	The verification is completely routine, one just checks everything satisfies all the definitions.
\end{proof}

\begin{nota}
	We will confuse the formal structure $\U[X]$ with anything isomorphic to it, and will imagine that $\U$ is literally an almost-initial-segment of $\U[X]$ instead of, merely, isomorphic to one.
	
	We will drop "$\top$-preserving" from "$\top$-preserving free extension of $\U$ by $X$" for brevity.
\end{nota}

\begin{defn}[Simple extension]
	If $\U$ and $\V$ are USL$^\top$s and $\V$ is an extension $\U$ generated over $\U$ by one element, then we call $\V$ a \textbf{simple extension} of $\U$.
\end{defn}

Observe that if $\U$ is finite, then any simple extension of $\U$ is also finite, as is any free extension by finitely many free generators.

\begin{thm}
	Let $\U$ be a finite USL$^\top$ and $\V$ a finite almost-end-extension. Then $\V$ is a subUSL$^\top$ of a simple almost-end-extension of a free extension of $\U$. Moreover, the free extension can be chosen with only finitely many free generators, and so the two extensions can be chosen to be finite.
\end{thm}

Jockusch and Slaman provide a proof of the above for USLs (without a named greatest element) and the corresponding notion of free and simple extension. Indeed, their proof applies to countable structures (although the extensions chosen must then be countable). Our proof follows theirs and also works in the countable case with little adjustment:

\begin{proof}
	Let $\U$ be an almost-initial-segment of a finite USL $\V$. Enumerate $\V\setminus \U = \{v_1,v_2,\ldots v_n\}$, let $A=\{a_1,\ldots,a_n\}$ be a set of new objects of the same size, and let $g$ be the map taking $v_i$ to $a_i$. We define $\U_1 = \U[A]$ the free extension of $\U$ by $A$, which is finite as both $\U$ and $A$ are finite. We define a map $h:\U_1\rightarrow \V$ by
		\[
			h(x) = \begin{cases}
				x_1 \join_\V \Join_\V \{v_i : g(v_i) = a_i \in A_1\}&\text{if $x = (x_1,A_1)$},\\
					\top & \text{if $x_1= \top$}.
			\end{cases}
		\]
	It is not hard to check that $h$ is a homomorphism from $\U_1$ to $\V$ (i.e., $h$ preserves least and greatest element, joins, and $\sqleq_\U$). Observe that if $x\in \U$ then $h(x) = x$ (considering $\U$ as an almost-initial-segment of $\U[A]$) and  if $v\in \V \setminus \U$ then $h(g(v)) = v$.
	
	Let $b$ be a new element and let $\U_2'$ be the free extension of $\U_1$ by $\{b\}$ (we write $(x,b)$ instead of $(x,\{b\})$). We define an equivalence relation $\equiv$ on $\U_2'$ by $y_0\equiv y_1$ iff $y_0 = y_1$, or
		\begin{align*}
			(\exists x_0,x_1\in \U_1)&[y_0 = (x_0,b)\text{ and } y_1 = (x_1,b) \text{ and }h(x_0) = h(x_1)],\\
			\text{or }& y_0 = \top_{\U_2'} \text{ and } (\exists x_1 \in \U_1)[y_1 = (x_1,b) \text{ and }h(x_1) = \top_\V],\\
			\text{or }& y_1 = \top_{\U_2'} \text{ and } (\exists x_0 \in \U_1)[y_0 = (x_0,b) \text{ and }h(x_0) = \top_\V].
		\end{align*}
	It is not hard to check that this is an equivalence relation, and that if $x_0,x_1\in \U_1$ then $x_0 \equiv x_1$ iff $x_0 =x_1$, indeed, the only element of $\U_1$ with a possibly nontrivial $\equiv$ equivalence class is $\top$.
	
	I claim that, further, $\equiv$ is also a congruence relation for the join structure of $\U_2'$, i.e., if $y_1,y_2,y_1',y_2'\in\U_2'$ and $y_1\equiv y_1', y_2\equiv y_2'$, then $y_1\join_{\U_2'} y_2 \equiv y_1'\join_{\U_2'} y_2'$. To prove the claim we break into cases:
	
	\emph{Case 1, none of the $y$s are  $\top$:} In this case each of the $y$s is of the form $(x,B)$ where $x\in \U_1$ and $B\subseteq \{b\}$, we write $y_1 = (x_1,B_1),y_2 = (x_2,B_2)$ and so on. Now
		\[
			y_1\join_{\U_2'} y_2 = (x_1\join_{\U_1}x_2,B_1\cup B_2)\text{ and }y_1'\join_{\U_2'} y_2' = (x_1'\join_{\U_1}x_2',B_1'\cup B_2').
		\]
	If each of the $B$s are empty, then $y_1=y_1'$ and $y_2= y_2'$ and so the joins displayed above are equal, and so trivially equivalent. Otherwise suppose WLOG that $B_1$ (and hence, $B_1'$) are nonempty.
	
	\emph{Subcase a, neither $x_1\join_{\U_1} x_2 = \top$ nor $x_1\join_{\U_1} x_2 = \top$:} In this case it suffices to show that $h(x_1\join_{\U_1} x_2) = h(x_1'\join_{\U_1} x_2')$. It has been observed already that $h$ is a homomorphism; therefore $h(x_1 \join_{\U_1} x_2) = h(x_1)\join_\V h(x_2)$. Our assumption that $y_1\equiv y_1',y_2\equiv y_2'$ implies that $h(x_1) = h(x_1')$ and $h(x_2) = h(x_2')$. Hence
		\[
			h(x_1 \join_{\U_1} x_2) = h(x_1)\join_\V h(x_2) = h(x_1')\join_\V h(x_2') = h(x_1'\join_{\U_1} x_2')
		\]
	which completes the subcase.
	
	\emph{Subcase b, $x_1 \join_{\U_1} x_2 = \top$:} It suffices to show that either $x_1'\join_{\U_1} x_2' = \top$ or that $h(x_1'\join_{\U_1}x_2')=\top$. If $x_1'\join_{\U_1} x_2' \neq \top$, then, as $h(x_1)=h(x_1'),h(x_2)= h(x_2')$ and $h$ is a homomorphism,
		\[
			h(x_1'\join_{\U_1} x_2')= h(x_1')\join_{\V}h(x_2') = h(x_1)\join_{\V}h(x_2') = h(x_1\join_{\U_1} x_2) = \top.
		\]
	as required. This also completes the case where $x_1'\join_{\U_1}x_2'= \top$.
	
	\emph{Case 2, $y_1 = \top$:} In this case $y_1\join_{\U_2'} y_2 = \top$, and so, it suffices to show that either $x_1'\join_{\U_1} x_2' =\top$ or $h(x_1'\join_{\U_1} x_2')=\top$. If $y_1'=\top$ then we are done. Otherwise $y_1'=(x_1',B'_1)$ and, as $x_1\equiv x_1'$, $h(x_1')=\top$. Therefore
		\[
			h(x_1'\join_{\U_1}x_2') = h(x_1')\join_\V h(x_2)' = \top
		\]
	which concludes this case, as well as the similar cases where one of the $y$s equals $\top$. Therefore, the claim is proved.
	
	With this, we define $\U_2 = \U_2'/\equiv$ and we equip it with $\join_{\U_2}$ defined by the action of $\join_{\U_2'}$ on the equivalence classes. We use this to induce a (suggestively notated) binary relation $\sqleq_{\U_2}$ on $\U_2$ by
		\[
			[y_1] \sqleq_{\U_2} [y_2] \text{ iff } [y_1] \join_{\U_2} [y_2] = [y_2] \text{ iff } y_1\join_{\U_2'} y_2 \equiv y_2.
		\]
	Using an easy yet tedious case analysis as above, it is not hard to see that this is a partial order on $\U_2$ and that $\join_{\U_2}$ is its join operator, and, furthermore, that $[\top_{\U_2'}]$ and $[\bot_{\U_2'}]$ are, respectively, the least and greatest elements. We omit the details.
	
	We want to show that $\U_2$ is a simple almost-end-extension of $\U_1$, or rather, we want to show the natural embedding of $\U_1$ into $\U_2$ given by $x\mapsto [(x,\emptyset)]$ realizes $\U_1$ as an almost-initial-segment and $\U_2$ is simple over this image. We already observed that there are no nontrivial $\equiv$ relationships between elements of $\U_1$, so the map is injective, and it also preserves $\top, \bot_{\U_1}, \sqleq_{\U_1}$ and $\join_{\U_1}$. We want to show it preserves $\not\sqleq_{\U_1}$. Well, if $x_1\not\sqleq_{\U_1} x_2$ then $x_1\join_{\U_1} x_2 \neq x_2$, and as there are no nontrivial $\equiv$ relationships between elements of $\U_1$ then $x_1\join_{\U_1}x_2 \not\equiv x_2$, so we preserve $\not\sqleq_{\U_1}$.
	
	$\U_2$ is clearly simple over this image of this embedding (as it is generated by $[(\bot,b)]$), so it remains to show $\U_1$ is an almost-initial-segment. Suppose $x \in \U_1$ is not $\top$ and $[y]\sqleq_{\U_2}[x]$ for some $y\in \U_2'$. We want to show $y\in \U_1$. By definition we have $y\join_{\U_2'}x \equiv x$. If $y\join_{\U_2'}x = x$, then $y\sqleq_{\U_2'}x$, and, as $\U_1$ is an almost-initial-segment of $\U_2'$, $y\in \U_2$. Otherwise, the $\equiv$ relation is nontrivial, but we have already observed that the $\equiv$ equivalence class of $x$ must be trivial, as $x\in \U_1$ but is not $\top$. So $\U_2$ is an almost-end-extension as required.
	
	So we have natural embeddings $\U$ into $\U_1$ into $\U_2$ both as almost-initial segments. We want to show this embedding of $\U$ into $\U_2$ extends to $\V$. The map we want is
		\[
			f(v) = \begin{cases}
				[\top]&\text{if $v = \top$},\\
				[((v,\emptyset),\emptyset)]&\text{if $v\in\U\setminus\{\top\}$},\\
				[(g(v),b)]&\text{if $v\in \V\setminus \U$}.
			\end{cases}
		\]
	Clearly this map extends the natural one of $\U$ into $\U_2$. We want to check that it is a USL$^\top$ embedding of $\V$. The verification of this is another long case analysis as provided previously. We provide the proof of the preservation of $\join$ when one of the elements is in $\U$ to guide the reader if they wish to check all the details.
	
	We have already shown that the action of $f$ on $\U$ is a USL$^\top$ embedding, so we only need to check that $f(v_1\join_\V v_2) = f(v_1)\join_{\U_2} f(v_2)$ when (WLOG) $v_1\in \V\setminus \U$ and $v_2\in \U$.
	
	\emph{Case 1, $v_1\join_\V v_2 = \top :$} In this case $f(v_1\join_\V v_2)= \top$ and so we want to show that $f(v_1)\join_{\U_2} f(v_2) = \top$. If $v_2 = \top$ we are done, so suppose otherwise. In this case
		\begin{align*}
		f(v_1)\join_{\U_2} f(v_2) & = [(g(v_1),b)]\join_{\U_2} [((v_2,\emptyset),\emptyset)]\\
		& = [(g(v_1),b)\join_{\U_2'}((v_2,\emptyset),\emptyset)]\\
		& = [(g(v_1)\join_{\U_1}(v_2,\emptyset),b)] 
		\end{align*}
	to show $(g(v_1)\join_{\U_1}(v_2,\emptyset),b) \equiv \top$ it suffices to show $h(g(v_1)\join_{\U_1}(v_2,\emptyset)) = \top$. But, of course,
		\begin{align*}
			h(g(v_1)\join_{\U_1}(v_2,\emptyset)) & = h(g(v_1))\join_{\V} h((v_2,\emptyset))\\
			&  =v_1 \join_{\V} v_2 \\
			& =\top,
		\end{align*}
	as required.
	
	\emph{Case 2, $v_1\join_\V v_2 \neq \top$:} As $V$ is an almost-end-extension of $\U$, $v_1\in \V\setminus \U$ and $v_1\join_\V v_2\neq \top$ then this join is strictly a member of $\V$. Hence, we want to show that $[(g(v_1\join_\V v_2),b)]=[(g(v_1),b)]\join_{\U_2} f(v_2)$. If $v_2\in \U\setminus\{\top\}$ (note it can't be equal to $\top$) then $f(v_2) = [((v_2,\emptyset),\emptyset)]$ and so
		\[
			[(g(v_1),b)]\join_{\U_2} f(v_2) = [(g(v_1),b)]\join_{\U_2} [((v_2,\emptyset),\emptyset)] = [(g(v_1)\join_{\U_1} (v_2,\emptyset),b)].
		\]
	Thus, we want to show that $g(v_1\join_\V v_2) \equiv g(v_1)\join_{\U_2'} (v_2,\emptyset)$, but quite similarly to the end of Case 1, $h(g(v_1\join_\V v_2)) = h(g(v_1)\join_\V (v_2,\emptyset)) = v_1 \join_{\U_2'} v_2$.
	
	We omit further details.
\end{proof}
In light of the previous Theorem, we only need to show Theorem \ref{maintheorem2} in the cases where $\V$ is a finite free extension or a simple extension of $\U$. The free extension results can be proved by Cohen forcing: Let $f:\U\rightarrow\DegO$ be any USL$^\top$ embedding and let $X$ be any Cohen real hyperarithmetic in $\O$ meeting every dense subset of Cohen forcing which is hyperarithmetic in $f(u)$ for any $u\in\Lat\setminus\{\top\}$. Then, the columns of $X$ are independent over each degree in the image of $\U\setminus\{\top\}$, and so, if $\V$ is freely generated over $\U$ by $v_1,\ldots,v_n$ then extend $f$ by mapping the generators $v_i \mapsto X^{[i]}$ and using the induced map on joins. This extends $f$ to a USL$^\top$ embedding of $\V$ into $\DegO$ (See Sacks Chapter IV Section 3 \cite{Sacks1990} for an exposition of hyperarithmetical Cohen forcing).

Now we turn to simple end extensions. We make a further reduction to an even more specialized case. The idea is that if a free extension has the fewest possible "positive" facts $x\leq y \join z$ possible, then we reduce the full simple end-extension case to allowing one new positive cupping fact to hold.

\begin{thm}[Bounded Posner-Robinson]
	Let $\a,\b,\c_i,\d_i$ be degrees in $\DegO$ for $i = 1,\ldots n$ and let $\e_j\in\Deg_h(<\O)$ for $j=1,\ldots,m$. Then the following holds
		\begin{align*}
			&\left(\bigwedge_{i=1}^n[\d_i\not\leq_h\c_i \,\&\,(\a\not\leq_h\c_i \text{ or } \d_i\not\leq_h \b\join_h\c_i)] \right)\\\rightarrow (\exists g<\O)&\left[\b\leq \a\join_h g \,\&\bigwedge_{i=1}^n\d_i\not\leq_h \c_i \join_h g \, \&  \bigwedge_{j=1}^m g\not\leq_h \e_j\right]
		\end{align*}
		\label{posnerrobinson}
\end{thm}
Barnes\cite{Barnes2017} has shown the above result without the bounds on the parameters and the generic object $g$. Our version, here, says that if the parameters are sufficiently bounded, then we can effect the forcing construction hyperarithmetically in $\O$.

Suppose the antecedent in the statement of Theorem \ref{posnerrobinson} and let $A\in \a, B\in \b, C_i\in \c_i, D_i\in \d_i$ and $E_j\in \e_j$ be representatives of the degrees. Barnes's construction uses Kumabe-Slaman forcing to produce a sequence of forcing conditions $\{\condition[i]:i\in\omega\}$ where each $\Phi_i$ is a finite use monotone Turing functional, and each $\X_i$ is a finite set of reals. We want to check that $\O$ can produce a generic of the correct kind.

The construction is complicated by the disjunction in the antecedent; If $\a\not\leq_h \c_i$ for the pair $(\c_i,\d_i)$ then we call $i$ an easy case, otherwise we call it a hard case. As there are only finitely many pairs, we can hard code into our algorithm which $i$s are easy and which are hard, so $\O$ need not be able to determine this uniformly.

The coding procedure for cupping $\a$ above $\b$ is to add axioms $(x,y,\alpha)$ to the generic object such that $\alpha\subset A$ and $B(x) = y$ (intuitively, when you plug $A$ into the generic functional $\Phi_g$ you get the characteristic function of $B$). Hence, the coding procedure is uniform and recursive in $A$ and $B$, and, as such, is recursive in $\O$.

For technical reasons, to the easy case we associate the usual Kumabe-Slaman forcing and a forcing language $\mathcal{L}_{\omega_1,\omega}^r(C_i)$, and to the hard case a restricted version of the forcing and the language $\mathcal{L}_{\omega_1,\omega}^r(B\oplus C_i)$. The restricted version of the forcing consists of all the conditions which do not explicitly get the coding procedure wrong. Importantly, the restriction is hyperarithmetic in $A\oplus B$ and so hyperarithmetic in $\O$. Additionally, the languages are recursive in $\omega_1^{C_i}$ and $\omega_1^{B\oplus C_i}$, respectively. As $C_i$ does not compute $D_i$ in the easy case and $B\oplus C_i$ does not compute $D_i$ in the hard case, neither of these are equal to $\O$, hence, their hyperjumps are hyperarithmetically equivalent to $\O$, and so the languages of forcing are hyperarithmetic in $\O$.
 
So, now we need to show that $\O$ can construct a generic sequence of conditions while maintaining the coding. We will construct one sequence $\{p_n = \condition[p_n]\}_{i=0}^\infty$ of Kumabe-Slaman conditions called the \textbf{master sequence}. To keep the complexity of the construction down, we need to make sure that the reals we add to the infinite part of a condition $\X$ are simple. This requires a little extra technical work; roughly speaking, say we are trying to meet a dense set corresponding to some fact we wish to force about $\Phi_g \oplus C_i$ and we have a condition $\condition[p_n]$ corresponding to what we have done so far. We will temporarily "forget" reals $X\in \X_{p_n}$ which are too complicated, to get a modified condition $(\Phi_{p_n},\X_{p_n}')$. We will find a simple extension of this condition which forces whichever fact we are trying to force, and then we will reinsert the "forgotten" reals. There is an obvious worry that this new condition need not refine $\condition[{p_n}]$, so we need to show that we can find an extension of $(\Phi_{p_n},\X_{p_n}')$ which doesn't add new axioms to $\Phi_{p_n}$ which apply to the reals we have forgotten.  This procedure is analogous to Barnes's Lemma 5.4 \cite{Barnes2017}. 

We start with the condition $p_0 = (\emptyset,\emptyset)$. We must decide each sentence of our forcing language. What we do depends on whether we are in the easy or hard case.  Note that, given a real $S$ strictly hyperarithmetic in $\O$ and a finite set $\X$ of reals each strictly hyperarithmetic in $\O$, that $\O$ can uniformly determine which of the $X\in \X$ are hyperarithmetic in $S$. Hence, the "forgetting" procedure mentioned above can be made effective in $\O$. We will preserve throughout that for each condition $p_n$ of our master sequence, each $X\in \X_{p_n}$ will be strictly hyperarithmetic in $\O$ (although their join may not be).

In the easy case Corollary 3.6 in Barnes\cite{Barnes2017} says that, given a sentence $\phi$ and a condition $(\Phi_p,\X_p)$, we can find an extension deciding $\phi$, without messing up the coding, uniformly in $A\oplus C_i^{(\alpha+1)}\oplus \X_p$ (where $\alpha$ is an ordinal less than $\omega_1^{C_i}=\omegaoneck$ which measures the complexity of $\phi$). So suppose we have a condition $p_n = \condition[p_n]$ such that $\Phi_n$ has coded correctly so far and $\X_{p_n}$ does not contain $A$. Let $p' = (\Phi_{p_n},\X_{p_n}')$ where $\X_{p_n}'$ is the intersection of $\X_{p_n}$ with the set of reals which are hyperarithmetic in $C_i$. We can extend Corollary 3.6 of Barnes\cite{Barnes2017} so that we can pick our extension to not add any axioms to any $X\in \X_{p_n} \setminus \X_{p_n}'$ very easily. The proof already does this for an arbitrary $S$ which is not $\Delta_0^{(\alpha+1)}$. The proof for finitely many such $S$ goes through in the same way, and as each $X\in \X_{p_n} \setminus \X_{p_n}'$ is not hyperarithmetic in $C_i$ we can apply this result to those reals. Thus, we produce an extension $q'$ of $p'$ which decides the sentence, such that $q = (\Phi_{q'}, \X_{q'} \cup \X_{p_n})$ extends $p_n$, does not interfere with the coding, and each $X\in \X_{q'}$ is hyperarithmetic in $C$. We can then define $p_{n+1} = q$.

In the hard case we rely on Corollary 4.5 of Barnes\cite{Barnes2017}(with similar modifications as in the easy case) to forget the reals not hyperarithmetic in $B\oplus C_i$, and find an extension deciding a sentence which is compatible with the starting condition.

We also need to diagonalize against cupping $C_i$ above $D_i$ (in the hard case we diagonalize against cupping $B\oplus C_i$ above $D_i$). The relevant results are Corollaries 3.11 and 4.8 in Barnes \cite{Barnes2017}, respectively. Although it is not observed directly in the statements of these Corollaries, it is clear from the proofs that the diagonalizing extension can be found uniformly in (some jump of) the previous condition (where the number of jumps needed is tied nicely to the complexity of the reduction against which we are diagonalizing). Consequently, our trick of temporarily forgetting reals which aren't hyperarithmetic in $C_i$ (or $B\oplus C_i$) allows us to prove that if we can't diagonalize against a reduction by forcing nontotality or for it to be incorrect on some fixed input, then we can hyperarithmetically in $C_i$ (or $B\oplus C_i$) recover what the current condition determines the outputs of this reduction to be. Hence, our assumption that $D_i$ is not hyperarithmetic in $C_i$ (or $B\oplus C_i$) means that any reduction we can't diagonalize against won't turn out to compute $D_i$.

Additionally, we must preserve $\omega_1^{C_i}$ or $\omega_1^{B\oplus C_i}$, depending on the case. For both, we force over nonstandard models $\M_i$ of ZFC, in particular, we force over countable $\omega$-models omitting $\omega_1^{CK}$, yet containing $C_i$ (or $B\oplus C_i$). Harrington, Shore, and Slaman \cite{HarringtonShoreSlaman2016} have shown that we can produce such models which are strictly hyperarithmetic in $\O$. As such, $\O$ can determine which $X\in \X_{p_n}$ are not in $\M_i$, and so can produce a modified condition $p_n'$ which has forgotten each real not appearing in the model. Furthermore, $\O$ can enumerate each element of the model which is a dense subset of (the model's version of) Kumabe-Slaman forcing. Then $\O$ can search for an extension of $p_n'$ in $\M_i$ which meets a dense set appearing in the model. By Lemma 5.4 of Barnes\cite{Barnes2017}, there is such an extension which adds no new computations to any of the reals we have forgotten, and so we can pick such an extension $q$ and extend the master sequence by defining $p_{n+1} = (\Phi_q, \X_{q}\cup \X_p)$, which does not interfere with the coding procedure.

Hence, for each $i$, the master sequence induces a sequence $p_n' = (\Phi_{p_n}, \X_{p_n} \cap (2^\omega)^{\M_i})$ such that each $p_n'\in \M_i$ and the sequence $\{p_n'\}$ is $\M_i$-generic for $\M_i$'s Kumabe-Slaman forcing. Thus, on general grounds, the generic object $\G'$ corresponding to $\{p_n\}$ preserves $\omegaoneck$ and, indeed, as $C_i\in \M_i$ we even have $C_i\oplus \G'$ preserves $\omega_1^{CK}$. Note, though, that the generic object $\G'$ is the same object as $\G$, the generic for the master sequence, and so, $C_i\oplus \G$ preserves $\omegaoneck$ as required.

Finally, we need to avoid ideals below $E_j$. Barnes \cite{Barnes2017} does not do this via genericity, instead he uses a counting argument. However, it is not hard to see that $\O$ can determine which sets are hyperarithmetic in $E_j$, and so attempt to diagonalize against our generic equaling these sets. This is certainly not difficult to do, but we must worry about interfering with our coding procedure.

Suppose we have a condition $(\Phi_p,\X_p)$, a set $Y$, and we are diagonalizing against $\Phi_g= Y$. We can assume that $Y$ is a use monotone Turing functional which is correct for $B$ on input $A$ (see Barnes \cite{Barnes2017} for definitions), as $\Phi_g$ will be such an object. Suppose that every $n$ we want to put into $\Phi_p$ (i.e. $n\notin Y$ but is allowed to enter $\Phi_p$) would interfere with our coding procedure, i.e., is of the form $(x,y,\alpha)$ with $\alpha\subset A$. As $\X_p$ does not contain $A$ (by induction) there is some sufficiently long initial segment of $A$ not an initial segment of any $X\in \X_p$ (and sufficiently long so as to not mess with use monotonicity, or that $\Phi_p$ is a Turing functional, and so on). Let $x$ be the least number such that there is no axiom about $x$ applying to $A$ in $\Phi_p$ (i.e., is the next value we need to code). As $Y$ is a use monotone Turing functional correct for $B$ on input $A$ there is only one axiom $(x,y,\alpha)\in Y$ with $\alpha\subset A$, so all we need to do is put in $(x,y,\alpha')$ where $\alpha'\subset A$ is sufficiently long to be allowed, and is not precisely $\alpha$. This information can all be determined uniformly in $\O$ and so we can diagonalize by meeting appropriate dense sets.

Consequently, we can produce a generic of the correct kind hyperarithmetically in $\O$ as required.

The last thing we need to do is show that Theorem \ref{posnerrobinson} implies we can extend embeddings to any simple almost-end-extension. But this is almost precisely Jockusch and Slaman's Theorem 3.1\cite{JockuschSlaman1993} with very minor changes to allow for the production of USL$^\top$ embeddings instead of USL embeddings (also, we should note that their proof makes use of allowing infinitely many $(\c_i,\d_i)$ and $\e_j$ but when your USLs are finite you only need arbitrary long finite lists).

\pagebreak


\begin{thebibliography}{9}
	
%

	\bibitem{Barnes2017}
	J. S. Barnes\\
	On the decidability of the $\Sigma_2$-theories of the arithmetic and hyperarithmetic degrees as uppersemilattices\\
	\emph{The Journal of Symbolic Logic}\\
	Forthcoming
	
	\bibitem{HarringtonShoreSlaman2016}
	L. A. Harrington, R. A. Shore, T. A. Slaman\\
	$\Sigma_1^1$ in every real in a $\Sigma^1_1$ class of reals is $\Sigma_1^1$\\
	\emph{Computability and Complexity}\\
	A. Day, M. Fellows, N. Greenberg, B. Khoussainov and A. Melnikov eds., Springer-Verlag to appear
	
	\bibitem{JockuschSlaman1993}
	C. G. Jockusch, T. A. Slaman\\
	On the $\Sigma_2$-theory of the upper semilattice of Turing degrees\\
	\emph{The Journal of Symbolic Logic}\\
	Vol. 58, Number 1, Association for Symbolic Logic, 1993
	
	\bibitem{KjoshanssenShore2010}
	B. Kjos-Hanssen, R. A. Shore\\
	Lattice Initial Segments of the Hyperdegrees\\
	\emph{Journal of Symbolic Logic}\\
	Vol. 75, No. 1, Association for Symbolic Logic, 2010 
	
	\bibitem{Lerman1983}
	M. Lerman\\
	Degrees of unsolvability\\
	\emph{Perspectives in Mathematical Logic}\\
	Omega Series, Springer-Verlag, Berlin, Heidelberg, New York, Tokyo, 1983
	
	\bibitem{LermanShore1988}
	M. Lerman, R. A. Shore\\
	Decidability and invariant classes for degree structures\\
	\emph{Transactions of the American Mathematical Society}\\
	Vol 310, No. 2, AMS, 1988
	
	\bibitem{Sacks1990}
	G. E. Sacks\\
	Higher Recursion Theory\\
	\emph{Springer-Verlag}\\
	1990
	


	
\end{thebibliography}
\end{document}